\documentclass[dvipsnames]{amsart}

\usepackage{amsfonts, amssymb, amsmath, amsthm, bm}
\usepackage{mathrsfs}                                                         
\usepackage{setspace}                                                         
\usepackage{geometry}                                                         
\usepackage{slashed}
\usepackage{xcolor}                                                           
\usepackage{graphicx}
\usepackage{caption}
\usepackage{subcaption}

\usepackage{cancel}
\usepackage{comment}                                                          
\usepackage{slashed}                                                          
\usepackage{physics}                                                          
\usepackage{cite}

\usepackage{tikz}
\usetikzlibrary{intersections}

\newtheorem{theorem}{Theorem}
\newtheorem{corollary}[theorem]{Corollary}
\newtheorem{lemma}[theorem]{Lemma}
\newtheorem{proposition}[theorem]{Proposition}
\newtheorem{definition}[theorem]{Definition}
\newtheorem{remark}[theorem]{Remark}
\newtheorem{assumption}[theorem]{Assumption}

\numberwithin{equation}{section}
\numberwithin{theorem}{section}

\newcommand{\ul}[1]{\underline{#1}}                                           
\newcommand{\mf}[1]{\mathfrak{#1}}                                            
\newcommand{\mc}[1]{\mathcal{#1}}                                             
\newcommand{\ms}[1]{\mathsf{#1}}                                              
\newcommand{\mi}[1]{\mathscr{#1}}                                             

\newcommand{\N}{\mathbb{N}}                                                   
\newcommand{\R}{\mathbb{R}}                                                   
\newcommand{\C}{\mathbb{C}}                                                   
\newcommand{\Sph}{\mathbb{S}}                                                 

\newcommand{\paren}[1]{\left(#1\right)}                                       
\newcommand{\floor}[1]{\lfloor{#1}\rfloor}                                    

\renewcommand{\grad}{\operatorname{grad}}                                     

\makeatletter
\newcounter{proofpart}
\newcommand{\proofpart}[1]{%
  \par
  \addvspace{\medskipamount}%
  \stepcounter{proofpart}%
  \noindent\textbf{\theproofpart. #1}\par\nobreak\smallskip
  \@afterheading
}
\makeatother

\allowdisplaybreaks
\begin{document}

\title[Counterexamples to unique continuation]{On counterexamples to unique continuation for critically singular wave equations}

\author{Simon Guisset}
\address{School of Mathematical Sciences\\
Queen Mary University of London\\
London E1 4NS\\
United Kingdom}
\email{s.guisset@qmul.ac.uk}

\author{Arick Shao}
\address{School of Mathematical Sciences\\
Queen Mary University of London\\
London E1 4NS\\
United Kingdom}
\email{a.shao@qmul.ac.uk}

\begin{abstract}
We consider wave equations with a critically singular potential $\xi \cdot \sigma^{-2}$ diverging as an inverse square at a hypersurface $\sigma = 0$.
Our aim is to construct counterexamples to unique continuation from $\sigma = 0$ for this equation, provided there exists a family of null geodesics trapped near $\sigma = 0$.
This extends the classical geometric optics construction \cite{AB} of Alinhac-Baouendi (i) to linear differential operators with singular coefficients, and (ii) over non-small portions of $\sigma = 0$, by showing that such counterexamples can be further continued as long as this null geodesic family remains trapped and regular.
As an application to relativity and holography, we construct counterexamples to unique continuation from the conformal boundaries of asymptotically Anti-de Sitter spacetimes for some Klein-Gordon equations; this complements the unique continuation results of the second author with Chatzikaleas, Holzegel, and McGill \cite{hol_shao:uc_ads, hol_shao:uc_ads_ns, Shao22, McGill20} and suggests a potential mechanism for counterexamples to the AdS/CFT correspondence.
\end{abstract}

\maketitle

\section{Introduction} \label{sec.intro}

We study the non-uniqueness of solutions to the singular geometric wave equation
\begin{equation}\label{operator}
\mathcal{P} u := \left[ \Box_g + \frac{\xi(\sigma, y)}{\sigma^2} \right] u = 0 \text{,}
\end{equation}
on a domain $\Omega := ( 0, \sigma_0 ) \times \mc{I}$, where $\sigma_0 > 0$ and $\mc{I}$ is an open subset of $\R^d$.
Here, $g$ and $\xi$ denote a Lorentzian metric and a bounded function on $\Omega$, respectively, while $\sigma$ and $y$ are the projections to the $( 0, \sigma_0 )$- and $\mc{I}$-components of $\Omega$.
In particular, $\mc{P}$ contains a critically singular potential that diverges at $\sigma = 0$, and which at leading order has the same scaling as $\Box_g$.

We investigate when \emph{unique continuation} for \eqref{operator} fails from $\sigma = 0$, that is, when solutions of \eqref{operator} fail to be uniquely determined by their Cauchy data on $\sigma = 0$.
More specifically, we show that counterexamples to unique continuation exist when there is an appropriate family of trapped null geodesics (or bicharacteristics) near $\sigma = 0$.
This will be accomplished through explicit geometric optics constructions similar to those of Alinhac--Baouendi \cite{AB}, resulting in solutions to \eqref{operator} that are supported throughout $\Omega$ but nonetheless vanish to infinite order at $\sigma = 0$.

In the context of wave equations, our results extend \cite{AB} in two ways:
\begin{enumerate}
\item While \cite{AB} only treated operators with smooth and bounded coefficients, here we consider operators $\mc{P}$ with a potential $\xi \sigma^{-2}$ that becomes singular at the hypersurface $\sigma = 0$.

\item While the counterexamples of \cite{AB} were only locally defined near a single point $p$, our constructions persist as long as the null geodesic family does not develop caustics.
\end{enumerate}
Furthermore, the methods here should extend to linear differential operators treated in \cite{AB}.

Our main motivation lies in holography, toward a potential mechanism for counterexamples to the AdS/CFT correspondence in theoretical physics.
More specifically, we apply our result to construct counterexamples to unique continuation for Klein--Gordon equations from the conformal boundaries of asymptotically anti-de Sitter (aAdS) spacetimes; see Section \ref{sec.intro_aads} for further discussions.

\subsection{Background} \label{sec.intro_bg}

Consider a linear differential operator $\mc{L}$ defined on a domain $\mc{U} \subseteq \R^{d+1}$, and fix a hypersurface $\Sigma \subseteq \mc{U}$.
The problem of \emph{unique continuation} for $\mc{L}$ from $\Sigma$ is to determine whether Cauchy data on $\Sigma$ uniquely determines solutions of $\mc{L} u = 0$ on one side of $\Sigma$, or equivalently, \emph{whether any solution of $\mc{L} u = 0$ with zero Cauchy data on $\Sigma$ must vanish identically on one side of $\Sigma$}.
Of particular interest here are settings in which the Cauchy problem for $\mc{L}$ from $\Sigma$ is ill-posed.

Unique continuation has been extensively studied over the past century; for brevity, here we focus only on results that are of direct relevance to the present article.
An important early result is Holmgren's theorem \cite{holmg:uc_anal} for operators $\mc{L}$ with analytic coefficients---in this case, unique continuation holds (even for distributional solutions) when $\Sigma$ is noncharacteristic.
Next, the modern theory---in particular the classical results of Calder\'on, Carleman, and H\"ormander \cite{Calderon1958, carleman1939, Hörmander2009}---extended the analysis to various classes of operators $\mc{L}$ with non-analytic coefficients.
However, in this more general setting, unique continuation for $\mc{L}$ holds only under the stronger condition that $\Sigma$ is (\emph{strongly}) \emph{pseudoconvex} (with respect to the principal symbol of $\mc{L}$ and the side of $\Sigma$ in which one continues the solution).
The intuitive role of pseudoconvexity is that it rules out the existence of bicharacteristics that can be roughly considered as ``locally trapped near $\Sigma$".

The seminal results of Alinhac and Baouendi \cite{alin:non_unique, AB} demonstrated that pseudoconvexity is crucial to unique continuation.
In the absence of pseudoconvexity for $\mc{L}$, both \cite{alin:non_unique, AB} used geometric optics to construct local counterexamples to unique continuation, that is, solutions propagating along the trapped bicharacteristics near $\Sigma$ while vanishing to infinite order on $\Sigma$.
However, a fundamental feature in their construction is that one cannot choose the precise operator $\mc{L}$ for which the counterexample applies.
More specifically, they only showed \emph{there exists a smooth potential $V$}, vanishing to infinite order at $\Sigma$, \emph{such that unique continuation for $\mc{L} + V$ from $\Sigma$ fails}.

Moreover, this potential $V$ is generally necessary, as unique continuation may still hold via Holmgren's theorem whenever $\mc{L}$ has analytic coefficients.
For this reason, \cite{AB} interpreted its construction as a \emph{zeroth-order instability} for Holmgren's unique continuation result.

We also note the counterexamples of \cite{alin:non_unique, AB} are \emph{local}, that is, defined in a sufficiently small neighbourhood of some $p \in \Sigma$.
This raises the question of \emph{whether such counterexamples can be extended over a larger portion of $\Sigma$, over which the above-mentioned trapped bicharacteristics persist}.

\vspace{0.4pc}
From here on, we narrow our focus to \emph{geometric wave operators},
\[
\mc{L} := \Box_g := | \det g |^{ -\frac{1}{2} } \partial_\alpha \big( | \det g |^\frac{1}{2} g^{\alpha\beta} \partial_\beta \big) \text{,}
\]
which can be viewed as a second-order hyperbolic operator.
(Here, $g$ is a smooth Lorentzian metric on $\mc{U}$, and $\partial$ denotes coordinate derivatives on $\mc{U}$.)
In practice, $\Sigma$ will be either a timelike or a null hypersurface, so that the Cauchy problem for $\Box_g$ is ill-posed from $\Sigma$.

As before, the Alinhac--Baouendi machinery produces localised counterexamples to unique continuation for $\Box_g + V$, for appropriate potentials $V$.
Here, the bicharacteristics for $\Box_g$ are precisely the null geodesics of $g$, hence the constructions lie along such geodesic trajectories near $\Sigma$.
Also, as mentioned in \cite{AB}, the zero-order instability is manifested, for instance, when $\mc{L}$ is the classical wave operator $\Box$ and $\Sigma$ is a timelike hyperplane (which is noncharacteristic but not pseudoconvex).

In this paper, we consider the wave operator $\mc{P}$ from \eqref{operator}, which contains an additional potential $\xi \sigma^{-2}$ that becomes singular on the entire hypersurface $\Sigma := \{ 0 \} \times \mc{I}$.
Moreover, as this potential has the same scaling as $\Box_g$, it must be treated as ``principal".
In particular, its presence radically alters the nature of the equation and the asymptotic behaviours of solutions at $\Sigma$.
(For example, when $\Sigma$ is timelike, the Dirichlet and Neumann branches gain specific powers of $\sigma$ at $\Sigma$; see \cite{Enciso19, Warnick13}.)

Note that the constructions of \cite{AB} cannot be directly applied to $\mc{P}$, as the methods crucially rely on the smoothness of the operator on $\Sigma$.
Thus, one goal of this article is to show that \emph{geometric optics counterexamples can nonetheless be extended to singular operators such as $\mc{P}$}.
In addition, another goal is to show that \emph{these counterexamples persist along the lifespan of the trapped null geodesics near $\Sigma$}, thereby addressing the above-mentioned question of locality arising from \cite{AB}.

Finally, we mention that singular operators of the form $\mc{P}$ naturally arise in aAdS spacetimes; see Section \ref{sec.intro_aads} below.
Their unique continuation properties, which are connected to the AdS/CFT correspondence, form the main motivation for this study.

\subsection{Statement of the theorem} \label{sec.intro_thm}

Throughout, we will work with the following setting:

\begin{definition} \label{def.domain}
We consider the domain
\begin{equation}
\label{domain} \Omega := ( 0, \sigma_0 ) \times \mc{I} \text{,} \qquad \mc{I} := \mc{I}' \times ( s_-, s_+ ) \text{.}
\end{equation}
where $\sigma_0 > 0$, $( s_-, s_+ )$ is a finite interval containing $0$, and $\mc{I}'$ is an open subset of $\R^{d-1}$.
Also:
\begin{itemize}
\item Let $\sigma$ and $y$ denote the projections onto the $( 0, \sigma_0 )$- and $\mc{I}$-components of $\Omega$, respectively.

\item Let $\bar{y} := ( y^1, \dots, y^{d-1} )$, $s := y^d$ denote further projections onto $\mc{I}'$, $( s_-, s_+ )$, respectively.

\item Let $\partial_\sigma$ and $\partial_s$ denote derivatives in the $\sigma$- and $s$-components, respectively.

\item Let $D$, $\nabla$, and $\bar{\nabla}$ denote derivatives with respect to $( \sigma, y )$, $y$, and $\bar{y}$, respectively.
\end{itemize}
\end{definition}

Next, we define the precise asymptotic and regularity properties for our geometric quantities:

\begin{definition}
Given $m \geq 0$ and a normed vector space $V$, we let $\mc{B}_m^\infty ( \Omega; V )$ denote the space of functions $\psi \in C^\infty ( \Omega; V )$ satisfying the following bounds for every $k, l \geq 0$: 
\begin{align}
\label{B_infty} \begin{cases}
\sup_\Omega | \partial_\sigma^k \nabla^l \psi | < \infty &\qquad k \leq m \text{,} \\
\sup_\Omega | \sigma^{k-m} \partial_\sigma^k \nabla^l \psi | < \infty &\qquad k > m \text{.}
\end{cases}
\end{align}
\end{definition}

\begin{remark}
To reduce technicalities, the reader may first substitute $\mc{B}^\infty_m ( \Omega; V )$ with, for instance, the space of functions for which all derivatives are uniformly bounded.
The rationale for using $\mc{B}^\infty_m ( \Omega; V )$ is due to our applications to aAdS spacetimes, in which $g$ can develop logarithmic singularities in $\sigma$ at higher orders, leading to bounds of the form \eqref{B_infty}; see \cite{fef_gra:conf_inv, shao:aads_fg}.

We also note, on the other hand, that the conditions \eqref{B_infty} are not optimal for our main result to hold.
The spaces $\mc{B}^\infty_m ( \Omega; V )$ merely serve as a compromise between simplicity of presentation in the upcoming proofs and applicability to the aAdS settings of interest.
\end{remark}

Our main result can now be stated as follows: 

\begin{theorem} \label{Theorem}
Let $d \geq 2$, let $\Omega$ be as in Definition \ref{def.domain}, and let $\mathcal{P}$ be the operator
\begin{equation}
\label{general_op} \mc{P} := \Box_g + \frac{\xi}{\sigma^2} \text{,} 
\end{equation}
with $g \in \mc{B}_1^\infty ( \Omega; \R^{(d+1) \times (d+1)} )$ a Lorentzian metric and $\xi \in \mc{B}^\infty_0 ( \Omega; \C )$.
In addition:
\begin{itemize}
\item Suppose there exist constants $C > 0$ and $\gamma \geq 0$ such that $g$ satisfies
\begin{equation}
\label{boundedness_metric} g^{-1} (d\sigma, d\sigma) \geq C \sigma^\gamma \text{,}
\end{equation}

\item Suppose $\varphi \in \mc{B}_2^\infty (\Omega; \R)$ satisfies the following on $\Omega$:
\begin{equation}
\label{conditions_s} g^{-1} ( d \varphi, d \varphi ) = 0 \text{,} \qquad 2 \, \operatorname{grad}_g \varphi = \partial_s \text{.}
\end{equation}
\end{itemize}
Then, there exist functions $u, a \in C^\infty(\Omega; \C)$ such that:
\begin{itemize}
\item All derivatives of $u$ and $a$ vanish faster than $\sigma^N$ as $\sigma \searrow 0$ for any $N \geq 0$.

\item $u$ is supported on all of $\Omega$.

\item $u$ solves the following equation on $\Omega$:
\begin{equation}
\label{singular_equation} \mc{P} u = a u \text{.}
\end{equation}
\end{itemize}
\end{theorem}

Note that Theorem \ref{Theorem} yields a counterexample to unique continuation for $\mc{P} - a$ from $\{ 0 \} \times \mc{I}$, via a solution $u$ of \eqref{singular_equation} that is everywhere nontrivial on $\Omega$ but has zero Cauchy data on $\sigma = 0$.
Also, by zero-extending the quantities of Theorem \ref{Theorem} to a small part of $\sigma < 0$, we can hence view Theorem \ref{Theorem} as the failure of unique continuation for $\mc{P} - a$ across $\Sigma := \{ 0 \} \times \mc{I}$.

Below, we discuss the intuitions behind the hypotheses \eqref{boundedness_metric}--\eqref{conditions_s}, which are closely related to those found in the main result of \cite{AB}:
\begin{itemize}
\item The assumption \eqref{boundedness_metric} implies that the level sets of $\sigma$ in $\Omega$ are timelike.
When $\gamma = 0$, then the boundary $\sigma = 0$ must be timelike as well.
On the other hand, if $\gamma > 0$, then the level sets of $\sigma$ could asymptote to a null boundary $\sigma = 0$.

\item The first condition in \eqref{conditions_s} states that $\varphi$ solves the eikonal equation with respect to $g$.
Thus, the level sets of $\varphi$ are null hypersurfaces, and the integral curves of the gradient $\operatorname{grad}_g \varphi$ generate a family $\mc{N}$ of null geodesics in $\Omega$ that propagate along these null hypersurfaces.

While solutions of the eikonal equation always exist locally, the smoothness of $\varphi$ can be viewed as an additional regularity assumption imposed on all of $\Omega$.
In particular, this condition implies that the null geodesics in $\mc{N}$ do not develop any caustics in $\Omega$.

\item Next, the second condition in \eqref{conditions_s}, which represents a convenient choice of gauge, aligns our coordinates $( \sigma, \bar{y}, s )$ so that the null geodesics of $\mc{N}$ are precisely given by
\[
s \mapsto \gamma (s) := ( \sigma_\ast, \bar{y}_\ast, s ) \text{,} \qquad ( \sigma_\ast, \bar{y}_\ast ) \in ( 0, \sigma_0 ) \times \mc{I}' \text{.}
\]
Note that the null geodesics in $\mc{N}$ lie on level sets of $\sigma$, and the coordinate $s$ serves as an affine parameter for these geodesics.
The interpretation is that $\mc{N}$ represents the ``trapped null geodesics" near $\sigma = 0$, discussed in Section \ref{sec.intro_bg}, propagating in a direction ``almost parallel" to $\sigma = 0$.
Moreover, $\mc{N}$ asymptotes, as $\sigma \searrow 0$, to limiting bicharacteristic curves (generated by $\partial_s$) on $\sigma = 0$; this can be viewed as $\{ 0 \} \times \mc{I}$ barely failing to be pseudoconvex.
\end{itemize}

\begin{figure}[ht]
\centering
\includegraphics[scale=0.9]{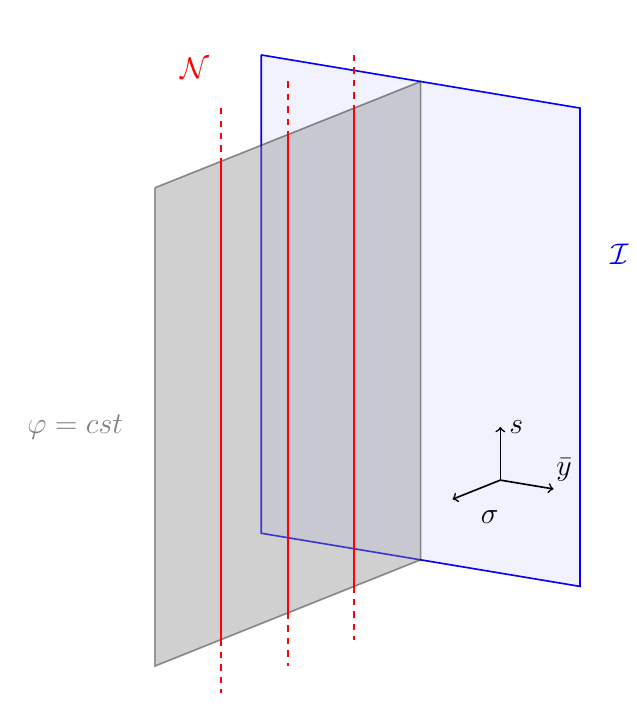}
\caption{Illustration of the setting of Theorem \ref{Theorem}, namely, of the geometry of $\Omega$ in our chosen $( \sigma, \bar{y}, s )$-coordinates.
The eikonal function $\varphi$ defines null hypersurfaces (in $\textcolor{gray}{grey}$) that intersect the boundary $\{ 0 \} \times \mc{I}$.
The gradient of $\varphi$ generates the null geodesic beams $\mc{N}$ (in \textcolor{red}{red}), along which our counterexample will propagate.}
\label{fig:example}
\end{figure}

\begin{remark}
The assumptions \eqref{boundedness_metric}--\eqref{conditions_s} can be viewed as direct extensions of the hypotheses of \cite{AB} to non-small boundary regions, at least in the context of geometric wave equations.
In particular, in \cite{AB}, the second part of \eqref{conditions_s} is replaced by a more general condition, $g^{-1} ( d \varphi, d \sigma ) = 0$.
However, since \cite{AB} only worked with sufficiently small neighbourhoods, one could always find a local change of coordinates such that the second part of \eqref{conditions_s} holds.
\end{remark}

While Theorem \ref{Theorem} is stated only for the critically singular wave equation \eqref{operator}, its proof should in principle extend to general higher-order linear differential operators treated in \cite{AB}---namely, those for which bicharacteristics are well-defined.
For this, the key step is to replace our geometric optics approximations with the more general analogues found within \cite{AB}.
Furthermore, using techniques introduced in this article, one can again allow for singular potentials behaving like powers of $\sigma$, and the counterexamples would extend for as long as the bicharacteristics persist.

Here, we opted to restrict our focus to geometric wave equations to simplify the presentation.
The specific form \eqref{operator} of the singular potential is taken due to its relevance to aAdS spacetimes (see Section \ref{sec.intro_aads}), and because such models are well-posed for a subclass of $\xi$; see \cite{Warnick13}.

Like in \cite{AB}, a key shortcoming of the approach in Theorem \ref{Theorem} is that the potential $a$ is part of the construction and cannot be freely chosen (although it does remain smooth up to $\sigma = 0$).
In particular, given any fixed (non-analytic) $a$, whether counterexamples to unique continuation exist for \eqref{operator} is generally not known.
It is also worth noting that all potentials $a$ and counterexamples $u$ constructed both in Theorem \ref{Theorem} and in \cite{AB} are genuinely complex-valued; whether real-valued $u$ and $a$ can be constructed in our context is currently unknown.

For similar reasons, whether analogous geometric optics counterexamples to unique continuation can be constructed for a fixed nonlinear wave equation is also a challenging open question.
Nonlinear settings also come with additional subtleties---for example, the now-classical result of M\'etivier \cite{Mtivier1993CounterexamplesTH} used the mechanism of \cite{AB} to construct counterexamples to unique continuation for some nonlinear PDEs, in particular showing that Holmgren's theorem fails to extend to nonlinear equations.

\subsection{Asymptotically anti-de Sitter spacetimes} \label{sec.intro_aads}

AdS/CFT can roughly be described as a one-to-one correspondence between a gravitational theory formulated on an asymptotically AdS background and a field theory on its conformal boundary; see \cite{Gubser_1998, Maldacena_1999, witten1998anti} for the seminal literature.
One rigorous formulation of the above, in the setting of classical relativity, is as a unique continuation problem for the Einstein-vacuum equations (EVE) from the conformal boundary---that is, establishing a one-to-one correspondence between aAdS solutions of the EVE and some space of (Cauchy) data on the conformal boundary.
Such results have been proven in time-independent contexts in \cite{and_herz:uc_ricci, and_herz:uc_ricci_err, biq:uc_einstein, chru_delay:uc_killing}; recently, Holzegel and the second author \cite{Holzegel22} proved a general result in dynamical settings, under additional geometric assumptions on the boundary.
See, e.g., the introduction of \cite{Holzegel22} for further references and for further details on this topic.

For aAdS spacetimes that solve the EVE, the components $\Psi$ of the Weyl curvature satisfy a system of wave equations of the form
\[
( \Box_g + m_\Psi ) \Psi = \text{l.o.t.,}
\]
for constants $m_\Psi \in \R$ that depend on the dimension and on the specific component $\Psi$.
(Here, ``l.o.t." refers to various nonlinear terms in $\Psi$ and its derivative.)
Thus, one linearised model formulation of the above problem is to consider a Klein--Gordon equation,
\begin{equation}
\label{aads_kg} ( \Box_{ \mathbf{g} } - m ) v = 0 \text{,} \qquad m \in \R \text{,}
\end{equation}
on a fixed aAdS spacetime $(\mc{M}, \mathbf{g} )$, for which $\mathbf{g}$ near its \emph{conformal boundary} $\mathscr{I}$ takes the form
\begin{equation}
\label{aads_metric} \mathbf{g} := \rho^{-2} [ d \rho^2 + \mathsf{g} ( \rho ) ] \text{,}
\end{equation}
where $\rho \in [ 0, \rho_0 )$ gives a measure of distance from the boundary, and $\mathsf{g} ( \rho )$ is a one-parameter family of metrics on $\mathscr{I}$.
Equation \eqref{aads_kg} has been studied from various perspectives \cite{BACHELOT2011527, breit_freedm:stability_sgrav, Holzegel12, Vasy09, Warnick13}, including boundary asymptotics, well-posedness, dynamics, and microlocal properties.

By considering instead the conformally transformed metric $g := \rho^2 \mathbf{g}$, which is now regular at the boundary, \eqref{aads_kg} now becomes equivalent to the \emph{critically singular} wave equation: 
\begin{equation}
\label{singular_kg} \big( \Box_g + C_m \rho^{-2} \big) u = \text{l.o.t.}
\end{equation} 
In particular, $C_m \neq 0$ whenever $4m \neq d^2-1$, and \eqref{singular_kg} becomes an equation of the form \eqref{operator}.
Applying a further change of parameter $\rho \mapsto \sigma$ (see Section \ref{sec.aads}), the operator \eqref{singular_kg} can be seen to also satisfy the conditions \eqref{general_op}--\eqref{boundedness_metric} of Theorem \ref{Theorem}.
As a result, Theorem \ref{Theorem} could be applied to generate counterexamples to unique continuation for \eqref{singular_kg}---or equivalently, \eqref{aads_kg} modulo a smooth potential---from the conformal boundary, provided a $\varphi$ satisfying \eqref{conditions_s} can be found.

In a series of articles \cite{hol_shao:uc_ads, hol_shao:uc_ads_ns, McGill20}, culminating in \cite{Shao22}, the authors establish unique continuation for \eqref{aads_kg} (plus additional lower-order terms) from a region $\mi{D}$ within the conformal boundary $\mi{I}$, provided $\mi{D}$ satisfies a geometric condition:\ the \emph{generalised null convexity condition} (\emph{GNCC}).
Roughly speaking, the GNCC can be interpreted as $\mi{D}$ being ``large enough" such that the above-mentioned ``trapped null geodesics" near the boundary cannot persist over all of $\mi{D}$.
Theorem \ref{Theorem}, in the aAdS context, then serves as a complement to the unique continuation results of \cite{Shao22}, in that if such trapped null geodesics (described by the eikonal function $\varphi$) do persist over all of $\mi{D}$, so that the GNCC fails, then unique continuation fails, and counterexamples can indeed be constructed.

In this way, Theorem \ref{Theorem} and the geometric optics construction behind its proof yield a mechanism for counterexamples to unique continuation for the linearised model \eqref{aads_kg} of AdS/CFT.
In addition, both the unique continuation results and the GNCC of \cite{Shao22} are crucial ingredients in \cite{Holzegel22} for proving unique continuation for the full EVE.
Consequently, the main result of this article could also serve as a potential mechanism for counterexamples to AdS/CFT.

In Section \ref{sec.aads}, we will illustrate, in detail, the above constructions of counterexamples in the special cases of pure AdS and planar AdS spacetimes.
For pure AdS, we extend our counterexamples over any time slab comprising less than one AdS cycle, while for planar AdS, we generate counterexamples that persist for an arbitrarily long timespan.
(In fact, the above-mentioned timespans are optimal in light of the unique continuation results of \cite{Shao22, hol_shao:uc_ads, hol_shao:uc_ads_ns, McGill20}.)
We also outline how this process can then be generalised to a larger class of asymptotically AdS spacetimes.

While the construction is relatively straightforward for planar AdS spacetime, the process is more complicated for pure AdS.
This is since the AdS conformal boundary has compact spherical cross-sections, which prevents $\varphi$ from being globally defined nearby.
As a result, Theorem \ref{Theorem} only yields counterexamples that are smoothly defined along a family $\mc{N}$ of null geodesics near a local sector of the conformal boundary.
To obtain counterexamples smoothly defined in a neighbourhood of the entire boundary, we must apply an additional cutoff that is carefully adapted to the family $\mc{N}$ of null geodesics; see Sections \ref{sec:planar} and \ref{sec:pure} for details.

Below, we give informal statements of our results for planar and pure AdS spacetimes, and we refer the reader to Sections \ref{sec:planar} and \ref{sec:pure} for the precise statements and proofs:

\begin{corollary} \label{corollary_planar_pre}
Let $(\mc{M}_{plan}, g_{plan})$ denote the planar AdS spacetime (see \eqref{aads_planar}), and fix $\mu \in \R$.
Then, for any $t_- < t_+$, there exists a smooth potential $V$ and a smooth counterexample $u$ to unique continuation for the Klein--Gordon equation
\[
( \Box_{ g_{plan} } + \mu ) u = Vu \text{,} 
\]
that is defined near the conformal boundary over the timespan $\{ t_- < t < t_+ \}$.
\end{corollary}

\begin{corollary} \label{corollary_pure_pre}
Let $(\mc{M}_{AdS}, g_{AdS})$ denote the pure AdS spacetime (see \eqref{aads_pure}), and fix $\mu \in \R$.
Then, for any $t_- < t_+$ satisfying $t_+ - t_- < \pi$, there exists a smooth potential $V$ and a smooth counterexample $u$ to unique continuation for the Klein--Gordon equation
\[
( \Box_{ g_{AdS} } + \mu ) u = Vu \text{,} 
\]
that is defined near the conformal boundary over the timespan $\{ t_- < t < t_+ \}$.
\end{corollary}

Finally, in light of AdS/CFT, it would be interesting to see to see whether this geometric optics mechanism could be used to build counterexamples to unique continuation for the (exact) Einstein-vacuum equations, $\operatorname{Ric}_g = -d \cdot g$, on aAdS settings in the absence of the GNCC of \cite{Shao22, Holzegel22}. 
A positive answer would complement the result of \cite{Holzegel22}, proving that the GNCC is the sharp criterion for unique continuation for the EVE from the boundary.

By the reasons stated at the end of Section \ref{sec.intro_thm}, the above would be a challenging question, for which major new ideas would be necessary.
However, we mention the recent articles of Touati \cite{toua:go_const, toua:go_eve, toua:go_wave}, which rigorously obtained---in asymptotically flat, well-posed settings---high-frequency solutions to the EVE and various semilinear wave equations through a geometric optics mechanism.

\subsection{Sketch of the construction} \label{sec:geom_optics}

We now turn our attention toward the proof of Theorem \ref{Theorem}.
Intuitively, the idea is the same as \cite{AB}---to propagate approximate solutions of \eqref{operator} along the family $\mc{N}$ of trapped null geodesics near $\sigma = 0$, as discussed in Section \ref{sec.intro_thm}.
Since an individual solution $v$ is chosen to be localised along a geodesic $\gamma \in \mc{N}$, which lies on level sets of $\sigma$, then $v$ by design does not contain any Cauchy trace at $\sigma = 0$.
By carefully combining a countable family of such $v$, one obtains an approximate solution supported on all of $\Omega$, but vanishing to all orders at $\sigma = 0$.

Let us first recall the geometric optics approximation. 
In our context, we consider
\begin{equation}
\label{geom_optics} \psi = e^{i\lambda \varphi} b \text{,} \qquad b := \sum_{k=0}^N b_k \lambda^{-k} \text{,}
\end{equation}
with $N, \lambda \gg 1$ sufficiently large.
Moreover, $\varphi$ is the eikonal function from Theorem \ref{Theorem} generating the null geodesic family $\mc{N}$, and the sequence $(b_k)_{k \leq N}$ of coefficients satisfy the following system of linear transport equations along the geodesics of $\mc{N}$:
\begin{align}
\label{set_transport_classic} i \paren{ \partial_s b_0 + \Box_g \varphi \, b_0} &= 0 \text{,} \\
\notag i \paren{ \partial_s b_k + \Box_g \varphi \, b_k} + \mc{P} b_{k-1} &= 0 \text{,} \quad 1 \leq k \leq N \text{.}
\end{align}

From the system \eqref{set_transport_classic}, we see that $\psi$ in \eqref{geom_optics} solves an approximate wave equation,
\begin{equation}
\label{approx_wave} \mc{P} \psi = e^{i\lambda \varphi} \lambda^{-N} \mc{P} b_N \text{,}
\end{equation}
and one can then hope to improve the approximation by increasing $\lambda$.
Notice that we terminate the sum \eqref{set_transport_classic} after a finite number of terms, as there is no expectation that the series $N \nearrow \infty$ converges.
Thus, an important ingredient will be to have uniform estimates of $b_N$ in terms of $\lambda$. 

To prove Theorem \ref{Theorem}, we largely follow the constructions from \cite{AB}, which we briefly summarise below.
The desired counterexample $u$ and error potential $a$ are defined to be
\begin{equation}
\label{counterexample} u := \sum_n \tilde{v}_n \text{,} \qquad a := \frac{\mc{P}u}{u} \text{,}
\end{equation}
where each beam $\tilde{v}_n := e^{ f_n } \psi_n$ consists of a geometric optics approximation $\psi_n$ as in \eqref{geom_optics}, with frequency $\lambda_n$ growing as a larger positive power of $n$, and supported in a strip $\sigma \sim n^{-1}$.
In addition, the amplitude of $\psi_n$ is attenuated by a factor $e^{ f_n }$ that vanishes exponentially at $\sigma = 0$.
Thus, by construction, $u$ indeed vanishes to infinite order as $\sigma \searrow 0$. Figure \ref{fig:example} gives a rough representation of how the geometric beams are transported along the trapped null geodesics. 

To complete the proof, it remains to show that $a$ in \eqref{counterexample} is well-defined on $\Omega$ and vanishes to infinite order at $\sigma = 0$.
As in \cite{AB}, to achieve the first task, we must address two technical issues:
\begin{enumerate}
\item One must control where $u$ vanishes on $\Omega$, which can occur when adjacent beams $\tilde{v}_n$, $\tilde{v}_{n-1}$ have equal amplitude.
Here, by adjusting the $f_n$'s, one ensures that $u$ can vanish only on a sequence $( S_n )$ of hypersurfaces in $\Omega$, for which any $S_n$, $S_{n-1}$ are sufficiently separated.

\item On each $S_n$, one must ensure that $\mc{P} u$ vanishes to higher order than $u$.
This is achieved by modifying the approximations $\psi_n$ to force $\mc{P} u = 0$ on each $S_n$.
The property \eqref{approx_wave} of $\psi_n$ is then crucial to ensure that the modification only occurs at high order in $\lambda_n$.
\end{enumerate}
Once $a$ is shown to be well-defined, one then uses that $u$ is comprised of successively improving geometric optics approximations, along with \eqref{approx_wave}, to show $a$ vanishes to any order as $\sigma \searrow 0$.

While the proof of Theorem \ref{Theorem} shares the same outline above as \cite{AB}, the key differences between our proof and that of \cite{AB} lie in how (1) and (2) are handled.
In particular, \cite{AB} makes crucial use of the fact that its operator $\mc{L}$ is smooth up to and including the boundary $\Sigma$.
This allows the authors to construct instead a continuous $1$-parameter family $( \tilde{w}_\delta )_\delta$ of geometric optics beams that also depend smoothly on $\delta \in [ 0, \delta_0 )$, from which they then extract the beam sequence $\smash{ \tilde{v}_n := \tilde{w}_{ \delta^{-1} } }$.
The smoothness at $\delta = 0$ leads to the necessary uniform control for the $\tilde{v}_n$'s.

In contrast, as our $\mc{P}$ is singular at $\sigma = 0$, we must take a different approach and instead construct our $\tilde{v}_n$'s directly.
Since this process avoids the boundary $\sigma = 0$ altogether, more care is now needed to ensure that the requisite control holds for the $\tilde{v}_n$'s uniformly in $n$.
Another price to be paid is that here we can only uniformly control a finite number of derivatives of each $\tilde{v}_n$; however, this will not pose a serious problem, as the number of derivatives becomes arbitrarily large as $n \nearrow \infty$.

With regards to (1), since \cite{AB} was only concerned with a small enough neighbourhood of a point, the authors could generate the hypersurfaces $S_n$  through the implicit function theorem.
As we want our counterexamples to persist on all of $\Omega$ (where $\sigma$ and $\varphi$ remain well-defined), we must be more refined in our approach.
Instead, by a direct analysis of $| \tilde{v}_n |$ and $| \tilde{v}_{n-1} |$ via our uniform estimates, we generate global hypersurfaces $S_n$ in $\Omega$ that contain all the vanishing points of $u$.
This, crucially, allows us to construct $u$ over all of $\Omega$, rather than only locally near a single point.

Regarding (2), the argument in \cite{AB} proceeds by treating $\delta$ in $( \tilde{w}_\delta )$ as an extra real variable.
This, along with the smoothness at $\delta = 0$, allows the authors to apply the Whitney extension theorem to construct modifications $\omega_n$ to the $\psi_n$'s that are uniformly controlled in $n$ and that achieve the requisite vanishing of $\mc{P} u$ at each $S_n$.
Here, we unfortunately cannot apply the Whitney theorem in the same manner, as we only have finite-order control for each of our $\tilde{v}_n$'s.
Instead, we construct the $\omega_n$'s directly by interpolating between its prescribed values on $S_n$ and $S_{n-1}$, and in a sufficiently careful manner so that the $\omega_n$'s can be adequately controlled uniformly in $n$.

The detailed proof of Theorem \ref{Theorem} will be provided in Section \ref{sec.const}.
This argument can roughly be divided into the following three main steps, each treated in an individual subsection:
\begin{itemize}
\item First, we define a preliminary version of the geometric optics beams $\tilde{v}_n$, and we obtain uniform bounds for the $\tilde{v}_n$'s using the system \eqref{set_transport_classic} of transport equations.

\item Next, we characterise the hypersurfaces $S_n$ from (1).
We then modify each beam $\tilde{v}_n$ to deal with (2), and we obtain uniform bounds for these modified beams.

\item Finally, we define $u$ and $a$ as in \eqref{counterexample}, and we show both vanish to all orders as $\sigma \searrow 0$.
\end{itemize}

\section{Proof of Theorem \ref{Theorem}} \label{sec.const}

Throughout this section, we assume the hyptheses of Theorem \ref{Theorem}, in particular the domain $\Omega$ from Definition \ref{def.domain}, the metric $g$, the eikonal function $\varphi$, as well as the function $\xi$.
Furthermore, we recall the coordinate and derivative conventions from Definition \ref{def.domain}:
\[
( \sigma, y ) := ( \sigma, \bar{y}, s ) \text{,} \qquad D := ( \partial_\sigma, \nabla ) := ( \partial_\sigma, \bar{\nabla}, \partial_s ) \text{.}
\]
Also, to keep track of dependencies in our estimates, we define the following family of constants: 

\begin{definition} \label{def_K_n}
Fix an increasing sequence $(K_N)_{N\geq 0}$ in $\R^+$, with each $K_N$ depending on $N$ and
\begin{align}
\label{eq_K_n} \sup_\Omega \bigg[ \sup_{ k + l \leq N + 1 } | \sigma^{ \max ( 0, k - 1 ) } \partial_\sigma^k \nabla^l g |, \,\, \sup_{ k + l \leq N } | \sigma^k \partial_\sigma^k \nabla^l \xi |, \,\, \sup_{ k + l \leq N + 2 } | \sigma^{ \max ( 0, k - 2 ) } \partial_\sigma^k \nabla^l \varphi | \bigg] \text{.}
\end{align}
In addition, we fix universal constants $\alpha \gg \beta > 0$, with $\alpha$ taken to be sufficiently large.
\end{definition}

Notice that the quantities in \eqref{eq_K_n} are finite by virtue of $g$, $\varphi$, $\xi$ being in the spaces $\mc{B}^\infty_1$, $\mc{B}^\infty_2$, $\mc{B}^\infty_0$, respectively.
The precise values of the constants $\alpha$, $\beta$, and $( K_N )$ will be determined throughout the proof. 
In particular, we will adjust these values as needed at various points in our construction.

\subsection{Geometric optics approximations}

Throughout, we fix a sufficiently large $n_0 \gg 1$, whose precise value is determined later.
The first step of the proof is to construct a sequence $( v_n )_{ n \geq n_0 }$ of geometric optics approximations, with each $v_n$ propagating along null geodesics at $\sigma \sim n^{-1}$.

To be more precise, we first decompose $\Omega$ into discrete bands:

\begin{definition}
For each $n \geq n_0$, we define the region
\begin{equation}
\label{def_Omega_n} \Omega_n := \left\{ \tfrac{1}{n+1} + \tfrac{1}{8(n+1)^2} < \sigma < \tfrac{1}{n-1} - \tfrac{1}{8(n-1)^2} \right\} \cap \Omega \text{.}
\end{equation}
Furthermore, on each $\Omega_n$, we define the following rescaled derivatives:
\begin{equation}
\label{def_D_n} \partial_{\sigma, n} := n^{-2} \partial_\sigma \text{,} \qquad D_n := ( \partial_{\sigma, n}, \nabla ) \text{,} \qquad \bar{D}_n := ( \partial_{\sigma, n}, \bar{\nabla} ) \text{.}
\end{equation}
\end{definition}

Observe in particular that for any $n, m \geq n_0$:
\begin{itemize}
\item $\Omega_n$ is contained in the region $\sigma \sim n^{-1}$. 
\item $\Omega_n \cap \Omega_m = \emptyset$ whenever $\abs{m - n} > 1$.
\end{itemize}
Moreover, we fix a particular sequence of constants, which will roughly reflect the number of terms we will take in our geometric optics approximation on each band $\Omega_n$:

\begin{proposition} \label{estimate_K_N}
There exists a non-decreasing sequence $( I_n )_{n \geq n_0} \subset \N$ such that:
\begin{itemize}
\item $\lim_{n \rightarrow +\infty} I_n = +\infty$.

\item For any $n \geq n_0$, we have $K_N \leq n$ for any $N \leq I_n$. 
\end{itemize}
\end{proposition}

\begin{proof}
As $( K_N )_{ N \geq 0 }$ is increasing, we can simply take $I_n := \max\{ N \mid K_N \leq n\}$ for any $n \geq n_0$. 
\end{proof}

Next, we define the following amplitudes for our geometric optics approximations on $\Omega_n$:

\begin{definition}
For each $n \geq n_0$, we define $f_n \in C^\infty ( \Omega_n )$ by
\begin{equation}
\label{f_n} f_n := -n^2 - n^4 \paren{\sigma - \tfrac{1}{n}} \cdot \theta \paren{n^2 (\sigma-\tfrac{1}{n}) } \text{,}
\end{equation}
where $\theta \in C^\infty ( \R )$ a smooth non-decreasing function satisfying
\begin{equation}
\label{def_theta} \theta(s) = \begin{cases} -\frac{1}{2} &\quad s \leq -\frac{1}{8} \text{,} \\ 1 &\quad s \geq \frac{1}{8} \text{.} \end{cases}
\end{equation}
\end{definition}

By direct computations, we have, for any $n \geq n_0$,
\begin{equation}
\label{estimate_f_n}
f_n(\sigma) \leq \begin{cases}
  -\frac{17}{16} n^2 &\quad \sigma - \frac{1}{n} \leq -\frac{1}{8n^2} \text{,} \\
  -\frac{7}{8} n^2 &\quad \abs{\sigma - \frac{1}{n}} \leq \frac{1}{8n^2} \text{,} \\
  -\frac{9}{8} n^2 &\quad \sigma - \frac{1}{n} \geq \frac{1}{8n^2} \text{,}
\end{cases}
\end{equation}
as well as the following bounds:
\begin{equation}
\label{estimate_derivative_f_n} \sup_{ \Omega_n } | \partial_{ \sigma, n } f_n | \leq K_0 n^2 \text{,} \qquad \sup_{ \Omega_n } | \partial_{ \sigma, n }^{N+2} f_n | \leq K_N n^2 \text{,} \quad N \geq 0 \text{.}
\end{equation}
Using the $f_n$'s, we then define the following operators:

\begin{definition}
Define the operator $T_1$ on $\Omega$ by
\begin{equation}
\label{def_T1} T_1 := \partial_s + \Box_g \varphi \text{.}
\end{equation}
In addition, for each $n \geq n_0$, we define the operator $T_{2, n}$ on $\Omega_n$ by
\begin{equation}
\label{def_T2} T_{2,n} := e^{-f_n} \mc{P} e^{f_n} \text{.}
\end{equation}
\end{definition}

Note in particular that the following identities hold for $n \geq n_0$:
\begin{align}
\label{T1_T2} in^{2\alpha}T_1 + T_{2,n} &= e^{-in^{2\alpha}\varphi} e^{-f_n} \mathcal{P} e^{f_n} e^{in^{2\alpha}\varphi} \text{,} \\
\notag T_{2, n} &= \Box_g + \frac{\xi}{\sigma^2} + 2 \partial_\sigma f_n \, \grad_g \sigma + g^{-1}(d\sigma, d\sigma) \big[ (\partial_\sigma f_n)^2 + \partial_\sigma^2 f_n \big] + \partial_\sigma f_n \, \Box_g \sigma \text{.}
\end{align}
The next step is to define the individual terms of our geometric optics approximations:

\begin{definition}
Fix a sequence $( \ul{\chi}_n )_{ n \geq n_0 }$ of smooth functions satisfying
\begin{equation}
\label{def_cutoff} \ul{\chi}_n: \R \rightarrow [0, 1] \text{,} \qquad \underline{\chi}_n (z) =
\begin{cases}
1 &\quad \frac{1}{n+1}+\frac{1}{6(n+1)^2} \leq z \leq \frac{1}{n-1} -\frac{1}{6(n-1)^2} \text{,} \\
0 &\quad z \leq \frac{1}{n+1}+\frac{1}{7(n+1)^2} \text{ or } z \geq \frac{1}{n-1}-\frac{1}{7(n-1)^2} \text{,}
\end{cases}
\end{equation}
and such that the following estimates hold:
\begin{equation}
\label{def_cutoff_est} \sup_{ z \in \R } | \ul{\chi}_n^{(N)} (z) | \leq K_N n^{2N} \text{,} \qquad N \geq 0 \text{.}
\end{equation}
Furthermore, for any $n \geq n_0$, we define the function
\begin{equation}
\label{def_chi} \chi_n \in C^\infty ( \Omega_n ) \text{,} \qquad \chi_n := \ul{\chi}_n ( \sigma ) \text{.}
\end{equation}
\end{definition}

Note that for $n$ large enough, the $\ul{\chi}_n$'s can be constructed as translations and dilations of a single function $\ul{\chi}: \R \rightarrow [0, 1]$.
The $\chi_n$'s will serve as initial profiles of our geometric optics terms.

\begin{definition}
For any $n \geq n_0$, we define the sequence $( c_{n, j} )_{ j \geq 0 }$ of functions on $\Omega_n$ to be the solutions to the following infinite system of transport equations:
\begin{align}
\label{def_c_0} T_1 c_{n,0} = 0 \text{,} &\qquad c_{n,0} |_{s=0} = \chi_n |_{s=0} = \ul{\chi}_n ( \sigma ) \text{,} \\
\label{def_c_j} i T_1 c_{n,j} + n^{-\alpha} T_{2, n} c_{n,j-1} = 0 \text{,} &\qquad c_{n,j} |_{s=0} = 0 \text{,} \quad j \geq 1 \text{.}
\end{align}
\end{definition}

In the following, we establish bounds for the $c_{n, j}$'s that capture the dependence on $n$ and $j$: 

\begin{proposition}
For any $n \geq n_0$, $j \geq 0$, and $N \geq 0$, the following estimates hold on $\Omega_n$:
\begin{align}
\label{estimate_c_0} K_0^{-1} \chi_n &\leq c_{n,0} \leq K_0 \chi_n \text{,} \\
\label{estimate_c_j} | D_n^N c_{n,j} | &\leq K_{N+2j} ( n^{8-\alpha}K_{N+2j} )^j | D^{\leq N+2j}_n \chi_n | \text{.}
\end{align}
\end{proposition}

\begin{proof}
Let us first observe that $T_{2,n}$ obeys the following estimate, for any $h \in C^\infty (\Omega_n)$: 
\begin{equation}
\label{estimate_T_2} | D^M_n T_{2,n} h | \leq K_M n^8 | D^{\leq M+2}_n h | \text{.}
\end{equation}
This follows from \eqref{T1_T2}, the estimates \eqref{estimate_derivative_f_n}, and our assumptions for the asympotics of $g$, $\xi$, and $\varphi$ (see Definition \ref{def_K_n}).
We now separate the proof into steps:

\vspace{\medskipamount}
\noindent
\textit{1. Estimate on $c_{n,0}$.}
In this step, we prove the following, for all $m, \ell \geq 0$: 
\begin{equation}
\label{estimate_c_0_1} \abs{\bar{D}_n^{\leq m} \partial_s^{\leq \ell} c_{n,0}} \leq K_{m+\ell} \abs{\bar{D}^{\leq m}_n \chi_n}.
\end{equation}

For this, we perform an induction on $\ell$---and for each $\ell$, another induction on $m$.
The base case, $\ell=0$, $m=0$ is proved by integrating \eqref{def_c_0} (noting \eqref{def_T1}) and applying the Gr\"onwall estimate:
\[
K_0^{-1} \chi_n \leq c_{n,0} \leq K_0 \chi_n \text{.}
\]
In particular, this proves the estimate \eqref{estimate_c_0}.

We now fix $m>0$ and assume \eqref{estimate_c_0_1} holds for $\ell=0$, and with $m$ replaced by $m-1$.
Differentiating \eqref{def_c_0} $m$ times with respect to $\bar{D}_n$ and using Gr\"onwall, we obtain
\begin{align*}
\abs{\bar{D}^m_n c_{n,0}} &\leq K_m \abs{\bar{D}^m_n \chi_n} + K_m \abs{\bar{D}^{\leq m-1}_n c_{n,0}} \\
&\leq K_m \abs{\bar{D}^{\leq m}_n \chi_n} \text{,}
\end{align*}
where we adjust $K_m$ between each step, if necessary, and where we used the induction hypothesis in the second line.
This proves \eqref{estimate_c_0_1} for the base case $\ell = 0$.

We can now fix $\ell > 0$ and assume \eqref{estimate_c_0_1} holds for all $m \geq 0$, and with $\ell$ replaced by $\ell-1$.
By applying $\partial_s^{\ell-1}$ and $\bar{D}_n^m$ to \eqref{def_c_0}, we then obtain
\begin{align*}
\abs{\bar{D}^m_n \partial_s^\ell c_{n,0}} &\leq K_{m+\ell} \abs{\bar{D}^{\leq m}_n \partial_s^{<\ell} c_{n,0}} \\
&\leq K_{m+\ell} \abs{\bar{D}^{\leq m}_n \chi_n} \text{,}
\end{align*}
where we used the induction hypothesis in the last line.
The full estimate \eqref{estimate_c_0_1} now follows.

\vspace{\medskipamount}
\noindent
\textit{2. Estimate on $c_{n,j}$ in terms of $c_{n,j-1}$.}
Next, we claim the following bound:
\begin{equation}
\label{intermediate_estimate} \abs{\bar{D}^{\leq m}_n \partial_s^{\leq\ell} c_{n,j}} \leq K_{m+\ell} n^{8-\alpha} \abs{D^{\leq m + \ell + 2}_n c_{n,j-1}} \text{,} \qquad m,\ell,j \geq 0 \text{.}
\end{equation}

To show this, we apply the same nested induction as before.
Let us first fix $j \geq 1$.
The base case $m=0$, $\ell=0$ of \eqref{intermediate_estimate} follows by integrating \eqref{def_c_j} and applying the Gr\"onwall estimate: 
\begin{align*}
\abs{c_{n,j}} &\leq K_0 n^{-\alpha} \abs{T_{2,n}c_{n,j-1}} \\
&\leq K_0 n^{8-\alpha} \abs{D_n^{\leq 2} c_{n,j-1}} \text{,}
\end{align*}
where we applied \eqref{estimate_T_2} with $M=0$.

Next, fix $m > 0$, and assume \eqref{intermediate_estimate} holds for $\ell=0$, and $m$ replaced by $m-1$.
Applying $\bar{D}_n^m$ to \eqref{def_c_j}, integrating in $s$, and then applying the Gr\"onwall estimate yields:
\begin{align*}
\abs{\bar{D}^m_n c_{n,j}} &\leq K_m \abs{\bar{D}^{\leq m-1}_n c_{n,j}} + K_m n^{-\alpha} \abs{\bar{D}^m_n T_{2,n} c_{n,j-1}} \\
&\leq K_m \abs{\bar{D}^{\leq m-1}_n c_{n,j}} + K_m n^{8-\alpha} \abs{D^{\leq m+2}_n c_{n,j-1}} \\
&\leq K_m n^{8-\alpha} \abs{D^{\leq m+2}_n c_{n,j-1}} \text{,}
\end{align*}
where we used \eqref{estimate_T_2} and the induction hypothesis.
This yields \eqref{intermediate_estimate} when $\ell = 0$.

We now perform the induction on $\ell$ by fixing $\ell>0$ and assuming \eqref{intermediate_estimate} to hold for any $m \geq 0$ and $\ell$ replaced by $\ell-1$.
Applying $\bar{D}_n^{\leq m} \partial_s^{\ell-1}$ to \eqref{def_c_j} then yields
\begin{align*}
\abs{\bar{D}^m_n \partial_s^\ell c_{n,j}} &\leq K_{m+\ell} \abs{ \bar{D}^{\leq m}_n \partial^{\leq \ell-1}_s c_{n,j}} + n^{-\alpha} \abs{\bar{D}^m_n \partial_s^{\ell-1} T_{2,n} c_{n,j-1}} \\
&\leq K_{m+\ell} \abs{\bar{D}^{\leq m}_n \partial^{\leq \ell-1}_s c_{n,j}} + K_{m+\ell} n^{8-\alpha} \abs{D^{\leq m+\ell+2}_n c_{n,j-1}} \\
&\leq K_{m+\ell} n^{8-\alpha} \abs{D^{\leq m+\ell+2}_n c_{n,j-1}} \text{,}
\end{align*}
where we again used \eqref{estimate_T_2} and the induction hypothesis.
This completes the proof of \eqref{intermediate_estimate}.

\vspace{\medskipamount}
\noindent
\textit{3. Induction on $j$.}
Finally, we derive \eqref{estimate_c_j} via an induction on $j$.
Note first the base case $j = 0$ is an immediate consequence of \eqref{estimate_c_0_1}.
Thus, for the remaining inductive case, we now fix $j > 1$, and we assume \eqref{estimate_c_j} holds with $j$ replaced by $j-1$.
Applying \eqref{intermediate_estimate}, one then has, for $N \geq 0$, 
\begin{align*}
\abs{D^N_n c_{n,j}} &\leq K_N n^{8-\alpha} \abs{D^{\leq N+2}_n c_{n,j-1}} \\
&\leq K_N n^{8-\alpha} K_{N+2+2(j-1)} \paren{ K_{N+2+2(j-1)} n^{8-\alpha} }^{j-1} \abs{ D_n^{\leq N+2+2(j-1)} \chi_n} \\
&\leq K_{N+2j} \paren{K_{N+2j} n^{8-\alpha}}^j \abs{ D_n^{N+2j} \chi_n } \text{,}
\end{align*}
where we used the induction hypothesis and adjusted the constants $K_{N+2j}$ as needed.
\end{proof}

In particular, adjusting the constants $\alpha$, $\beta$, $( K_N )_{ N \geq 0 }$ as needed, \eqref{estimate_c_j} implies:

\begin{corollary}
For any $n \geq n_0$, $N \geq 0$, and $j \geq 0$ such that $N + 2j \leq I_n$ (see Proposition \ref{estimate_K_N}),
\begin{equation}
\label{beta} \abs{D^N_n c_{n,j}} \leq n^\beta n^{j(\beta-\alpha)} \abs{D^{\leq N+2j}_n \chi_n} \text{.}
\end{equation}
\end{corollary}

\begin{remark}
In particular, Proposition \ref{estimate_K_N} allows us to remove constants of the form $K_N$, with the price of adding a factor $n^\beta$, as long as $\beta$ is sufficiently large.
This is a trick that we will use multiple times throughout this section, as a matter of convenience, to remove $K_N$'s from inequalities, and we will do so in the following development without further mention.
\end{remark}

We can now define the preliminary family of geometric optics approximation bands:

\begin{definition}
For any $n \geq n_0$, we define the following functions on $\Omega_n$: 
\begin{align}
\label{c_star} c_{n,\star} &:= \sum_{j=1}^{I_n^\star} n^{-j\alpha} c_{n,j} \text{,} \qquad I_n^\star := \Big\lfloor \tfrac{I_n}{3} \Big\rfloor \text{,} \\
\label{v_n} v_n &:= e^{in^{2\alpha}\varphi}e^{f_n} \paren{c_{n,0}+c_{n,\star}} \text{.}
\end{align}
\end{definition}

\begin{proposition}
The following holds for any $n \geq n_0$ and $N \leq I_n^\star$:
\begin{align}
\label{estimate_c_star} \abs{D^{\leq N}_n c_{n,\star}} &\leq K_{N+2I_n^\star} n^{\beta-\alpha} n^{-\alpha} \abs{D^{\leq N + 2I_n^\star}_n \chi_n} \\
\notag &\leq n^{\beta - 2 \alpha} \abs{D_n^{\leq I_n} \chi_n} \text{.}
\end{align}
Furthermore, the following holds for all $N\leq I_n^\star-2$:
\begin{align}
\label{prop_estimate_sum} \abs{D^N_n\left[\paren{in^{2\alpha} T_1 + T_{2,n}}(c_{n,0}+c_{n,\star})\right]} &\leq n^\beta n^{2(\beta-\alpha)I_n^\star} \abs{D_n^{\leq N+2+2I_n^\star} \chi_n} \\
\notag &\leq n^\beta n^{2(\beta-\alpha)I_n^\star} \abs{D_n^{\leq I_n} \chi_n} \text{.}
\end{align}
\end{proposition}

\begin{proof}
The first part of \eqref{estimate_c_star} follows by applying \eqref{beta} to each term of the summation in \eqref{c_star}.
The second part of \eqref{estimate_c_star} then follows by adjusting the constants $\alpha$, $\beta$, $( K_N )_{ N \geq 0 }$ as needed and recalling Proposition \ref{estimate_K_N}.
(Note in particular that $N + 2 I_n^\star \leq I_n$.)

Next, for \eqref{prop_estimate_sum}, we first expand $c_{n,\star}$ using \eqref{c_star}, and we note that
\begin{align*}
\paren{ in^{2\alpha} T_1 + T_{2,n} }(c_{n,0} + c_{n,\star}) &= n^{2\alpha} i T_1 c_{n,0} + \sum_{ j = 1 }^{ I_n^\star } n^{(2-j) \alpha} \paren{i T_1 c_{n,j} + n^{-\alpha} T_{2,n} c_{n,j-1}} \\
&\qquad + n^{-I_n^\star \alpha} T_{2,n} c_{n,I_n^\star} \\
&= n^{-I_n^\star \alpha} T_{2,n} c_{n,I_n^\star} \text{,}
\end{align*}
where we used the transport equations \eqref{def_c_0}, \eqref{def_c_j} in the last step.
Differentiating the above and recalling \eqref{estimate_c_j} and \eqref{estimate_T_2}, we then obtain the bound
\begin{align*}
\abs{D^N_n \paren{in^{2\alpha}T_1 + T_{2,n}} \paren{c_{n,0}+c_{n,\star}} } &= n^{-I_n^\star\alpha} \abs{D^N_n T_{2,n} c_{n,I_n^\star}} \\
&\leq n^{-I_n^\star \alpha + \beta} \abs{D^{\leq {N+2}}_n c_{n,I_n^\star}} \\
&\leq n^{-I_n^\star \alpha + \beta} n^{I_n^\star(\beta - \alpha)} \abs{D_n^{\leq N + 2 + 2I_n^\star} \chi_n} \text{,}
\end{align*}
from which the first part of \eqref{prop_estimate_sum} follows.
(As before, we use powers $n^\beta$ to absorb various constants $K_M$.)
The remaining inequality in \eqref{prop_estimate_sum} now follows immediately, since $N + 2 + 2 I_n^\star \leq I_n$.
\end{proof}

\subsection{Modification of bands}

The desired counterexample $u$ will be constructed as a sum of all the geometric optics bands $v_n$.
However, the key problem is that the corresponding potential ($a$ in Theorem \ref{Theorem}) could become singular where $u$ vanishes; in particular, this could occur wherever $| v_n |$ and $| v_{n+1} |$ are the same.
In this subsection, we modify each $v_n$ such that the modified bands $\tilde{v}_n$, $n \geq n_0$, avoid the above-mentioned issue---that is, they satisfy the following:
\begin{itemize}
\item Each $\tilde{v}_n$ satisfies roughly the same estimates as $v_n$.

\item Both $\mc{P} \tilde{v}_n$ and $\mc{P} \tilde{v}_{n+1}$ vanish when $| \tilde{v}_n | = | \tilde{v}_{n+1} |$.
\end{itemize}

\begin{definition}
For any $n \geq n_0$, we define the set
\begin{equation}
\label{def_S_n} S_n := \{ x \in \Omega_n \cap \Omega_{n+1} \mid \abs{v_n (x)} = \abs{v_{n+1} (x)} \} \text{.}
\end{equation}
\end{definition}

\begin{proposition}\label{interference_prop}
If $n \geq n_0$, then $S_n$ is a smooth graph in $\Omega_n \cap \Omega_{n+1}$ of the form
\begin{equation}
\label{graph_S_n} S_n = \lbrace (\sigma, y) \in \Omega_n \cap \Omega_{n+1} \mid y \in \mc{I} \text{, } \sigma = \mf{s}_n (y) \rbrace \text{,}
\end{equation}
where $\mf{s}_n \in C^\infty ( \mc{I} )$ is of the form,
\begin{equation}
\label{frak_s_n} \mf{s}_n (y) = n^{-1} - \tfrac{2}{3} n^{-2} + C_n^1 n^{-3} + \tilde{\mf{s}}_n (y) \text{,} \qquad \abs{C^1_n} \leq K_0 \text{,}
\end{equation}
and where $\tilde{\mf{s}}_n \in C^\infty ( \mc{I} )$ obeys the following estimates:
\begin{equation}
\label{tilde_frak_s_n} \abs{ \nabla^{\leq N} \tilde{\mf{s}}_n } \leq n^{\beta - 2 \alpha} \text{,} \qquad N \leq I_n^\star \text{.}
\end{equation}
\end{proposition}

\begin{proof}
We separate the proof into two steps:

\vspace{\medskipamount}
\noindent
\textit{1. Existence and uniqueness of $\mf{s}_n$.}
Consider the function
\begin{equation}
\label{Phi_n} \Phi_n := \log \abs{\frac{v_n}{v_{n+1}}} = f_n - f_{n+1} + \log \abs{\frac{c_{n,0} + c_{n,\star}}{c_{n+1,0} + c_{n+1,\star}}} \text{,}
\end{equation}
so that $S_n$ is simply the level set $\{ \Phi_n = 0 \}$.
Noting that (for large enough $n_0$)
\[
(n+1)^{-1} + \tfrac{1}{8} (n+1)^{-2} \leq \sigma \leq n^{-1} - \tfrac{1}{8} n^{-2}
\]
on $\Omega_n \cap \Omega_{n+1}$, and expanding \eqref{f_n}, we have
\begin{align*}
\Phi_n &= -n^2 + \tfrac{1}{2} n^4 \paren{\sigma - n^{-1}} + (n+1)^2 + (n+1)^4 \paren{\sigma - (n+1)^{-1}} \\ 
&\qquad + \log \abs{c_{n,0} + c_{n,\star}} - \log \abs{c_{n+1,0}+c_{n+1,\star}} \text{,}
\end{align*}
which can be rewritten as
\begin{align}
\label{form_phi_n} B_n^{-1} \Phi_n &= \sigma - \paren{n^{-1} - \tfrac{2}{3}n^{-2} + C^1_n n^{-3}} + C^2_n n^{-4} \log \abs{c_{n,0}+c_{n,\star}} \\
\notag &\qquad -C ^2_n n^{-4} \log \abs{c_{n+1,0}+c_{n+1,\star}} \text{,}
\end{align}
with constants $| C_n^1 | \leq K_0$, $0 \leq C_n^2 \leq K_0$, and $B_n := \frac{3}{2}n^4 + 4n^3 + 6n^2 + 4n + 1$.
We will show here that $\lbrace \Phi_n = 0 \rbrace$ will correspond to a smooth graph $\sigma = \mf{s}_n$ lying in the region $\chi_n = \chi_{n+1} = 1$.  

First, consider the case $\sigma \leq (n+1)^{-1} + \frac{1}{6}(n+1)^{-2}$.
Since $\chi_n \leq 1$ and $\chi_{n+1} = 1$ in this case (see \eqref{def_chi}), we can bound, using \eqref{estimate_c_0} and \eqref{estimate_c_star},
\begin{align*}
-\log \abs{c_{n+1,0}+c_{n+1,\star}} &\leq n^{\beta - 2\alpha} \text{,} \\
\log \abs{c_{n,0}+c_{n,\star}} &= \log \abs{c_{n,0}} + \log \abs{1+\frac{c_{n,\star}}{c_{n,0}}} \\
&\leq K_0 + n^{\beta - 2\alpha} \text{.}
\end{align*}
(Note since $\chi_n < 1$, we only obtain an upper bound for $\log | c_{n,0} + c_{n,\star} |$.)
Thus, by \eqref{form_phi_n}, we have
\begin{align*}
B_n^{-1} \Phi_n &\leq \paren{(n+1)^{-1} + \tfrac{1}{6}(n+1)^{-2}} - \paren{n^{-1} - \tfrac{2}{3}n^{-2} + C_n^1 n^{-3}} + C^2_n n^{-4} ( K_0 + n^{\beta - 2 \alpha} ) \\
&\leq - \tfrac{1}{12} (n+1)^{-2} \text{,}
\end{align*}
as long as $n_0$ is sufficiently large.
In particular, $\Phi_n < 0$ when $\sigma \leq (n+1)^{-1} + \frac{1}{6}(n+1)^{-2}$.

Similarly, if $\sigma \geq n^{-1}-\frac{1}{6}n^{-2}$, then $\chi_n = 1$ and $\chi_{n+1} \leq 1$, and \eqref{estimate_c_0} and \eqref{estimate_c_star} yield
\begin{align*}
\log \abs{c_{n,0}+c_{n,\star}} &\geq - n^{\beta - 2\alpha} \text{,} \\
- \log \abs{c_{n+1,0}+c_{n+1,\star}} &= - \log \abs{c_{n+1,0}} - \log \abs{1+\frac{c_{n+1,\star}}{c_{n+1,0}}} \\
&\geq - K_0 - n^{\beta - 2\alpha} \text{.}
\end{align*}
(Note we now only have a lower bound for $- \log \abs{c_{n+1,0}+c_{n+1,\star}}$.)
Thus, by \eqref{form_phi_n},
\begin{align*}
B_n^{-1} \Phi_n &\geq \paren{n^{-1} - \tfrac{1}{6} n^{-2}} - \paren{n^{-1} - \tfrac{2}{3}n^{-2} + C_n^1 n^{-3}} - C^2_n n^{-4} ( K_0 + n^{\beta - 2 \alpha} ) \\
&\geq \tfrac{1}{12} n^{-2} \text{,}
\end{align*}
if $n_0$ is sufficiently large.
Thus, $\Phi_n > 0$ when $\sigma \geq n^{-1}-\frac{1}{6}n^{-2}$.

Lastly, for the remaining case $(n+1)^{-1} + \frac{1}{6}(n+1)^{-2} < \sigma < n^{-1} - \frac{1}{6}n^{-2}$, we have $\chi_n = \chi_{n+1} = 1$.
Differentiating \eqref{form_phi_n} and again applying both \eqref{estimate_c_0} and \eqref{estimate_c_star}, we can bound
\begin{align*}
B_n^{-1} \partial_\sigma \Phi_n &= 1 + C^2_n n^{-2} \paren{n^{-2} \partial_\sigma} \log \abs{\frac{c_{n,0}+c_{n,\star}}{c_{n+1,0}+c_{n+1,\star}}} \\
&\geq 1 - 2 n^{-2} n^{\beta - 2\alpha} \text{,}
\end{align*}
which is strictly positive for sufficiently large $n_0$.
Since $\Phi_n < 0$ on $\sigma = (n+1)^{-1} + \frac{1}{6}(n+1)^{-2}$, $\Phi_n > 0$ on $\sigma = n^{-1}-\frac{1}{6}n^{-2}$, and $\partial_\sigma \Phi_n > 0$ in the above-mentioned region, it follows that for each $y \in \mc{I}$, there is a unique $\mf{s}_n (y) \in ( (n+1)^{-1} + \frac{1}{6}(n+1)^{-2}, \, n^{-1} - \frac{1}{6}n^{-2} )$ such that $\Phi_n ( \mf{s}_n (y), y ) = 0$.
In particular, this yields the desired characterisation \eqref{graph_S_n} of $S_n$.

\vspace{\medskipamount}
\noindent
\textit{2. Estimates on $\mf{s}_n$.}
For convenience, let us define
\begin{equation}
\label{h_n} h_n := C^2_n \log \abs{ c_{n,0} + c_{n,\star} } - C^2_n \log \abs{ c_{n+1,0} + c_{n+1,\star} } \text{,} 
\end{equation}
Since $\Phi_n ( \mf{s}_n (y), y ) = 0$ for every $y \in \mc{I}$, then by \eqref{frak_s_n}, we can write \eqref{form_phi_n} as
\begin{align}
\label{frak_sn_pre} \tilde{\mf{s}}_n (y) &= \mf{s}_n (y) - \paren{n^{-1} - \tfrac{2}{3}n^{-2} + C^1_n n^{-3}} \\
\notag &= n^{-4} h_n ( \mf{s}_n (y), y ) \text{.}
\end{align}
In particular, the above, along with \eqref{estimate_c_0} and \eqref{estimate_c_star}, implies the estimate
\[
\abs{\tilde{\mf{s}}_n (y)} \leq n^{\beta - 2\alpha } \text{,}
\]
where recalled that $\chi_n = \chi_{n+1} = 1$ on $\sigma = \mf{s}_n (y)$; this is precisely \eqref{tilde_frak_s_n} with $N = 0$.

For $N > 0$, we differentiate \eqref{h_n} repeatedly and then estimate
\begin{align*}
\abs{ \nabla^N \tilde{\mf{s}}_n } &\leq K_N n^{-4} \sum_{ \substack{ M+Q+P=N \\ Q_1+\dots+Q_M = Q \\ Q \leq N-1} } \abs{ \partial_\sigma^M \nabla^P h_n } \prod_{k=1}^M \abs{ \nabla^{Q_k} \nabla \mf{s}_n } \\
&\leq K_N n^{-2} \abs{ n^{-2} \partial_\sigma h_n } \abs{ \nabla^N \tilde{\mf{s}}_n } + K_N n^{-4} \sum_{ \substack{ M+P+Q=N \\ Q_1+\dots+Q_M = Q \\ Q<N-1 } } \abs{ \partial_\sigma^M \nabla^P h_n } \prod_{k=1}^M \abs{ \nabla^{Q_k} \nabla \tilde{\mf{s}}_n} \\
&\leq n^{-2} n^{\beta - 2\alpha} \abs{ \nabla^N \tilde{\mf{s}}_n } + n^{-4} \sum_{ \substack{ M+P+Q=N \\ Q_1+\dots+Q_M = Q \\ Q<N-1 } } n^{2M} n^{\beta - 2\alpha} \prod_{k=1}^M \abs{ \nabla^{Q_k} \nabla \tilde{\mf{s}}_n} \text{,}
\end{align*}
where the derivatives of $h_n$ are again controlled using \eqref{estimate_c_0} and \eqref{estimate_c_star}.
Since $\alpha \gg \beta$, the first term in the right-hand side can be absorbed in the left-hand side, giving
\[
\abs{ \nabla^N \tilde{\mf{s}}_n } \leq n^{-4} \sum_{ \substack{ M+P+Q=N \\ Q_1+\dots+Q_M = Q \\ Q<N-1 } } n^{2M} n^{\beta - 2\alpha} \prod_{k=1}^M \abs{ \nabla^{Q_k} \nabla \tilde{\mf{s}}_n} \text{.}
\]
The desired estimate \eqref{tilde_frak_s_n} now follows from the above by an induction on $N$.
(Note that in the base case $N = 0$, the right-hand side does not contain any factor involving $\tilde{\mf{s}}_n$.)
\end{proof}

The computations in the preceding proof also yield the following bounds:

\begin{corollary}\label{prop_sign_f_n}
The following estimates hold for any $n \geq n_0$:
\begin{itemize}
\item $f_{n+1} - f_n \leq K_0$ in the region $\{ ( \sigma, y ) \in \Omega_n \cap \Omega_{n+1} \mid \sigma \geq \mf{s}_n (y) \}$.
\item $f_{n} - f_{n+1} \leq K_0$ in the region $\{ ( \sigma, y ) \in \Omega_n \cap \Omega_{n+1} \mid \sigma \leq \mf{s}_n (y) \}$.
\end{itemize}
\end{corollary}

\begin{proof}
We focus here mainly on the first inequality.
In the case $\mf{s}_n (y) \leq \sigma \leq n^{-1} - \frac{1}{6} n^{-2}$, letting $\Phi_n$ be as in \eqref{Phi_n} and applying \eqref{estimate_c_0} and \eqref{estimate_c_star}, we obtain
\begin{align*}
f_{n+1}-f_n &= -\Phi_n + \log\abs{c_{n,0}+c_{n,\star}} - \log\abs{c_{n+1,0}+c_{n+1,\star}} \\
&\leq K_0 \text{,}
\end{align*}
as desired.
In particular, we noted (see the proof of Proposition \ref{interference_prop}) that $\Phi_n \geq 0$ in this region.
On the other hand, if $\sigma \geq n^{-1} - \frac{1}{6} n^{-2}$, we can simply apply the definition \eqref{f_n}:
\begin{align*}
f_{n+1}-f_n &= -(n+1)^2 - (n+1)^4 \paren{\sigma - (n+1)^{-1}} + n^2 - \tfrac{1}{2} n^4 \paren{\sigma - n^{-1}} \\ 
&\leq -2n - 1 - (n+1)^4 \paren{ n^{-1} - \tfrac{1}{6} n^{-2} - (n+1)^{-1} } - \tfrac{1}{2} n^4 \paren{ n^{-1} - \tfrac{1}{6} n^{-2} - n^{-1}} \\
&\leq -K_0 n^2 \text{.}
\end{align*}
Finally, the second inequality can be similarly proved by again separating into two cases:
\[
(n+1)^{-1} + \tfrac{1}{6} (n+1)^{-2} \leq \sigma \leq \mf{s}_n (y) \text{,} \qquad \sigma \geq (n+1)^{-1} + \tfrac{1}{6} (n+1)^{-2} \text{.} \qedhere
\]
\end{proof}

We now construct the modification $\omega_n$ that we will apply to $v_n$ in order to define the transformed band $\tilde{v}_n$, in particular so that both $\mc{P} \tilde{v}_n$ and $\mc{P} \tilde{v}_{n+1}$ vanish on $S_n$:

\begin{proposition} \label{omega_n}
For any $n \geq n_0$, there exists $\omega_n \in C^\infty ( \Omega_n )$ satisfying the following:
\begin{itemize}
\item $\omega_n \vert_{S_n \cup S_{n-1}} = \partial_\sigma \omega_n \vert_{S_n \cup S_{n-1}} =0$.

\item $\omega_n$ solves, on $S_n \cup S_{n-1}$: 
\begin{equation}
\label{equation_omega_n} \left. D_n^N \left[ \paren{in^{2\alpha} T_1 + T_{2,n}} (c_{n,0}+c_{n,\star}+\omega_n) \right] \right|_{S_n \cup S_{n-1}} = 0 \text{,} \qquad N \leq I_n^{\star\star} := \Big\lfloor \tfrac{I_n^\star-2}{2} \Big\rfloor \text{.}
\end{equation}

\item $\operatorname{supp} \omega_n \subset \operatorname{supp} \chi_n$.

\item The following estimates hold on $\Omega_n$:
\begin{equation}
\label{estimate_omega_n} \abs{D^N_n \omega_n} \leq n^\beta n^{(\beta-\alpha)I_n^{\star\star}} \text{,} \qquad N \leq I_n^{\star \star} \text{.}
\end{equation}
\end{itemize}
\end{proposition}

\begin{proof}   
We first introduce a harmless change of coordinates $( \sigma, y ) \mapsto ( \eta, \tilde{y} )$ as follows:
\begin{equation}
\label{change_variable} \eta(\sigma, y) := \frac{ n^{-2} \paren{\sigma - \mf{s}_n(y)} }{\mf{s}_{n-1} (y) - \mf{s}_n (y) } \text{,} \qquad \tilde{y} (\sigma, y) := y \text{.}
\end{equation}
With this set of coordinates, the sets $S_n$ and $S_{n-1}$ take the simple form: 
\begin{equation}
\label{Sn_tilde} S_n = \lbrace x \in \Omega_n \cap \Omega_{n+1} \mid \eta (x) = 0 \rbrace \text{,} \qquad S_{n-1} = \lbrace x \in \Omega_{n-1} \cap \Omega_n \mid \eta (x) = n^{-2} \rbrace \text{.}
\end{equation}
In addition, we denote derivatives in $\tilde{y}$-coordinates by $\tilde{\nabla}$, and we set
\begin{equation}
\label{tilde_deriv} \tilde{D} := (\partial_\eta, \tilde{\nabla}) \text{,} \qquad \partial_{\eta, n} := n^{-2} \partial_\eta \text{,} \qquad \tilde{D}_n := (\partial_{\eta,n}, \tilde{\nabla}) \text{.}
\end{equation}

We now claim the Jacobian associated to \eqref{change_variable} satisfies, on $\Omega_n$,
\begin{equation}
\label{jacobian} K_0^{-1} \leq \abs{\frac{\partial(\eta, \tilde{y})}{\partial(\sigma, y)}} \leq K_0 \text{.}
\end{equation}
For this, the non-trivial derivatives to estimate are
\[
\partial_\sigma \eta = \frac{n^{-2}}{\mf{s}_{n-1}-\mf{s}_n} \text{,} \qquad \nabla \eta = - \frac{n^{-2} \nabla \tilde{\mf{s}}_n}{ \mf{s}_n - \mf{s}_{n-1} } - \frac{ \eta \cdot \nabla \paren{\tilde{\mf{s}}_{n-1} - \tilde{\mf{s}}_n} }{ \mf{s}_n - \mf{s}_{n-1} } \text{.}
\]
From the above, along with \eqref{frak_s_n} and \eqref{tilde_frak_s_n}, we immediately estimate
\[
K_0^{-1} \leq \partial_\sigma \eta \leq K_0 \text{,} \qquad \abs{\nabla \eta} \leq n^{\beta - 2\alpha} \leq n^{-2} \text{,}
\]
which immediately implies \eqref{jacobian}.
Note that \eqref{jacobian} yields that for any $h \in C^\infty (\Omega_n)$: 
\begin{equation}
\label{derivative_eta} K_N^{-1} \abs{D^{\leq N}_n h} \leq \abs{\tilde{D}{}^{\leq N}_n h} \leq K_N \abs{D^{\leq N}_n h} \text{,} \qquad \abs{\tilde{\nabla}{}^{\leq N} h} \leq K_N \abs{D^{\leq N}_n h} \text{.}
\end{equation}

\vspace{\medskipamount}
\noindent
\textit{1. Data for $\omega_n$ on $S_n \cup S_{n-1}$.}
Informally, the relation \eqref{equation_omega_n} (for $N = 0$) can be written on $S_n \cup S_{n-1}$ as: 
\[
g^{-1} (d\eta, d\eta) \, \partial_\eta^2 \omega_n =n^{2\alpha} \paren{ \mathbf{q}_{n, 0, 1} \cdot \partial_{\eta} \tilde{\nabla}^{\leq 1} \omega_n + \mathbf{q}_{n,0,2} \cdot \tilde{\nabla}^{\leq 2} \omega_n } - \paren{in^{2\alpha} T_1 + T_{2,n}}\paren{c_{n,0} + c_{n,\star}} \text{,}
\]
with $\mathbf{q}_{n, M, k}$ appropriate smooth and bounded vector-valued functions obeying: 
\[
\abs{\mathbf{q}_{n, M, k}} \leq K_{M+k} \text{,} \qquad M + k\leq I_n^{\star \star} \text{.}
\]
By differentiating, one obtains, for each $M \geq 0$, the recursive relation
\[
g^{-1} (d\eta, d\eta) \, \partial^{M+2}_{\eta} \omega_n = n^{2 \alpha} \sum_{k=0}^M \sum_{\ell = 1, 2} \bm{q}_{n, k, \ell} \cdot \partial_\eta^{M+2-k-\ell} \tilde{\nabla}^{\leq \ell} \omega_n - \partial^M_{\eta} \paren{in^{2\alpha} T_1 + T_{2,n}}(c_{n,0} + c_{n,\star}) \text{.}
\]

With the above as motivation, we define the data $\smash{h^{(j)}_{n, M}}$ (with $j=0, 1$) on $S_{n-j}$ as
\begin{align}
\label{definition_data} ( h^{(j)}_{n,0}, h^{(j)}_{n,1} ) &:= ( 0, 0 ) \text{,} \\
\notag g^{-1} (d\eta, d\eta) \, h^{(j)}_{n,M+2} (\tilde{y}) &:= n^{2\alpha} \sum_{k=0}^M \sum_{\ell = 1, 2} \mathbf{q}_{n,k,\ell} \paren{\eta = jn^{-2}, \tilde{y}} \cdot \tilde{\nabla}^{\leq \ell} h_{n, M+2-k-\ell} \paren{\tilde{y}} \\
\notag &\qquad - \partial_{\eta}^M \paren{in^{2\alpha}T_1 + T_{2,n}} \paren{c_{n,0}+c_{n,\star}} (\eta = jn^{-2}, \tilde{y}) \text{,}
\end{align}
for every $M \leq I^{\star \star}_n$.
In particular, $\smash{h^{(0)}_{k,n}}$ corresponds to data for the transverse derivatives $\partial_{\eta, n}^{k+2} \omega_n$ evaluated on $S_n$, and similarly for the $\smash{h^{(1)}_{k, n}}$'s and $S_{n-1}$. 
We now claim the $( \smash{h_{n,M+2}^{(j)}} )$'s are smooth and bounded for each $j=0,1$ and $M\leq I_n^{\star\star}$, and that for any $N \leq I_n^\star - M - 2$:
\begin{align}
\label{estimate_data} \abs{\tilde{\nabla}^N h^{(j)}_{n,M+2}} \leq n^\beta n^{2(\beta - \alpha) I_n^\star} n^{2M\alpha} \text{.}
\end{align}

To prove \eqref{estimate_data}, we first note from \eqref{boundedness_metric} and the computations behind \eqref{jacobian} that
\[
\abs{g^{-1} (d\eta, d\eta)} \geq K_0^{-1}n^{-\gamma} \text{.}
\]
(In particular, from \eqref{definition_data}, the $h^{(j)}_{ n, M+2 }$'s are well-defined and smooth.)
We now proceed by induction on $M$---first, for $M = 0$, we apply \eqref{definition_data} and the above to estimate, for any $N \leq I_n^\star - 2$,
\begin{align*}
\abs{\tilde{\nabla}^N h^{(j)}_{n,2}} &\leq K_{N+2} n^\gamma \abs{D^N_n \paren{in^{2\alpha}T_1+T_{2,n}} \paren{c_{n,0}+c_{n,\star}} } \\
&\leq K_{N+2} n^\beta n^{2(\beta-\alpha)I_n^\star} \text{,}
\end{align*}
where we also made use of \eqref{prop_estimate_sum} and \eqref{derivative_eta}, and where in the last step, the contribution of $n^\gamma$ was absorbed by adjusting $\beta$.
(Recall in particular that $\chi_n=1$ on both $S_n$ and $S_{n-1}$.)
Next, fix $M \leq I_n^{\star\star}$, and assume \eqref{estimate_data} holds for any $\tilde{M}<M$ in the place of $M$.
Then,
\begin{align*}
\abs{\tilde{\nabla}^N h^{(j)}_{n,M+2}} &\leq K_{ N + M + 2 } n^\beta n^{2\alpha} \sum_{k = 1}^M \abs{\tilde{\nabla}^{\leq N+k} h^{(j)}_{n,M+2-k}} \\
&\qquad + n^{2M} n^\beta \abs{D_n^{\leq N+M} \paren{in^{2\alpha} T_1 + T_{2,n}} \paren{c_{n,0}+ c_{n,\star}}} \\
&\leq K_{M+N+2} n^\beta \left[ n^{2\alpha} n^{2(\beta-\alpha) I_n^\star} n^{2(M-1)\alpha} + n^{2M} n^\beta n^{2(\beta - \alpha)I_n^\star} \right] \\
&\leq K_{M+N+2} n^\beta n^{2(\beta-\alpha)I_n^\star} n^{2M\alpha} \text{,}
\end{align*}
for any $N \leq I_n^\star - M - 2$, where we also applied \eqref{prop_estimate_sum} and the induction hypothesis.
Adjusting the constants $K_{M+N+2}$ and $\beta$ as needed results in the desired estimate \eqref{estimate_data}.

\vspace{\medskipamount}
\noindent
\textit{2. Construction of $\omega_n$.}
We can now define $\omega_n$ on $\Omega_n$ by
\begin{equation}
\label{expansion_omega} \omega_n (\eta, \tilde{y}) := \sum_{j=0,1} \zeta(n^2(\eta-jn^{-2})) \sum_{0 \leq k \leq I_n^{\star\star}} h^{(j)}_{k+2,n}(\tilde{y}) \, \tfrac{\paren{\eta - jn^{-2}}^{k+2}}{(k+2)!} \text{,}
\end{equation}
where $\zeta$ is a smooth cutoff satisfying
\[
\zeta(s) = \begin{cases}
  1 &\quad \abs{s} \leq \frac{\varepsilon}{2} \text{,} \\
  0 &\quad \abs{s} \geq \varepsilon \text{,}
\end{cases}
\qquad \varepsilon > 0 \text{.}
\]
(In particular, $\omega_n$ matches the data $( h^{(j)}_{n, M} )$ up to order $I_n^{\ast\ast} + 2$ on $S_{n-j}$, and smoothly extends this data to all of $\Omega_n$.)
Notice that if $\varepsilon$ is sufficiently small, then $\omega_n$ satisfies the first three properties in the statement of Proposition \ref{omega_n} by construction.

Thus, it remains only to prove the bound \eqref{estimate_omega_n}, which is a consequence of \eqref{estimate_data}.
For simplicity, we only consider the terms $j=0$ in the right-hand side of \eqref{expansion_omega}.
(The $j=1$ terms are similarly controlled.)
In this case, given any $m + \ell \leq I_n^{\star\star}$, we obtain:
\begin{align*}
\abs{\partial_{\eta,n}^m \tilde{\nabla}^\ell \sum_{k=0}^{I_n^{\star\star}} \frac{ \eta^{k+2} }{(k+2)!} h^{(0)}_{n,k+2} \, \zeta\paren{n^2\eta} } &\leq K_{ I_n^{\star\star} } \sum_{k=0}^{I_n^{\star\star}} \abs{\tilde{\nabla}^\ell h^{(0)}_{n,k+2}} \abs{\partial_{\eta,n}^m \paren{\zeta(n^2\eta) \eta^{k+2}}} \\
&\leq n^\beta n^{2(\beta-\alpha) I_n^\star} n^{2 I_n^{\star\star} \alpha} \\
&\leq n^\beta n^{2(\beta-\alpha) I_n^{\star\star}} \text{,}
\end{align*}
where we applied \eqref{estimate_data} and recalled the definition of $I_n^{\star\star}$ in \eqref{equation_omega_n}.
Since analogous estimates hold for the $j=1$ terms, \eqref{estimate_omega_n} finally follows from the above and \eqref{derivative_eta}.
\end{proof}

Using the $\omega_n$'s, we can now define our modified geometric optics bands $\tilde{v}_n$:

\begin{definition}
For any $n \geq n_0$, we define the function $\tilde{v}_n$ on $\Omega_n$ by
\begin{equation}
\label{v_n_tilde} \tilde{v}_n := v_n + e^{in^{2\alpha}\varphi} e^{f_n} \omega_n = e^{in^{2\alpha}\varphi}e^{f_n}\paren{c_{n,0}+c_{n,\star}+\omega_n} \text{.}
\end{equation}
\end{definition}

\begin{proposition}\label{prop_properties_v_n_tilde}
The following holds for any $n \geq n_0$:
\begin{equation}
\label{supp_v_n_tilde} \operatorname{supp} \tilde{v}_n \subset \operatorname{supp} \chi_n \text{.}
\end{equation}
Furthermore, for any $\mu \geq 0$ and $N \geq 0$,
\begin{equation}
\label{decay_v_n} \lim_{n \rightarrow +\infty} n^\mu \sup_{\Omega_n} \abs{D^N \tilde{v}_n} = 0 \text{.}
\end{equation}
\end{proposition}

\begin{proof}
\eqref{supp_v_n_tilde} is immediate, as by construction, $c_{n, 0}$, $c_{n,\star}$, and $\omega_n$ are all supported within $\operatorname{supp} \chi_n$.
The property \eqref{decay_v_n} follows from the estimates \eqref{estimate_f_n}, \eqref{estimate_c_0}, \eqref{estimate_c_star} and \eqref{estimate_omega_n}, which yield
\[
\sup_{\Omega_n} \abs{D^N \tilde{v}_n} \lesssim n^{2N\alpha} e^{-Cn^2} \text{,} \qquad n \gg_{ n_0, N } 1 \text{.} \qedhere
\]
\end{proof}

\subsection{Construction of the counterexample}

We now construct our desired counterexample $u$ in Theorem \ref{Theorem} by ``gluing together" each of the modified bands $\tilde{v}_n$:

\begin{definition}
We define the following functions on $\Omega$,
\begin{equation}
\label{ua} u := \sum_{n \geq n_0} \tilde{v}_n \text{,} \qquad a := \frac{ \mc{P} u }{u} \text{,}
\end{equation}
where each $\tilde{v}_n$ is defined to be zero outside of $\Omega_n$.
\end{definition}

\begin{proposition}
$u \in C^\infty ( \Omega )$, and the following holds for any $N, \mu > 0$:
\begin{equation}
\label{u_est} \lim_{ \sigma_0 \searrow 0} \sup_{ \lbrace \sigma = \sigma_0 \rbrace } \abs{\sigma^{-\mu} D^N u} = 0 \text{.}
\end{equation}
\end{proposition}

\begin{proof}
By definition, one of the following must hold for any $x \in \Omega_n$, with $n > n_0$:
\begin{enumerate}
\item $x$ is in the support of only exactly one $\tilde{v}_n$. 
\item $x$ lies in the supports of $\tilde{v}_n$ and $\tilde{v}_{n+1}$, for some unique $n$.
\end{enumerate}
The estimate \eqref{u_est} follows immediately from \eqref{decay_v_n} and the above, since $\sigma^{-1} \sim n$ on each $\Omega_n$.
\end{proof}

In particular, \eqref{u_est} yields the desired vanishing for the counterexample $u$ in Theorem \ref{Theorem}.
Now, it remains only to prove the requisite properties of the potential $a$ from \eqref{ua}.

\begin{proposition}
The following holds for any $n \geq n_0$:
\begin{equation}
\label{S_tilde_n} \tilde{S}_n := \lbrace x \in \Omega \mid \abs{\tilde{v}_n (x)} = \abs{\tilde{v}_{n+1}(x)} \rbrace = S_n \text{.}
\end{equation}
Furthermore, the following estimate holds on $\Omega_n \cap \Omega_{n+1}$:
\begin{equation}
\label{proposition_s_n_tilde} \abs{\tilde{v}_n + \tilde{v}_{n+1}} \geq \begin{cases}
\abs{\tilde{v}_n} \cdot K_0 \min(\sigma - \mf{s}_n, 1) &\quad \sigma \geq \mf{s}_n \text{,} \\
\abs{\tilde{v}_{n+1}} \cdot K_0 \min(\mf{s}_n - \sigma, 1) &\quad \sigma \leq \mf{s}_n \text{.}
\end{cases}
\end{equation}
\end{proposition}

\begin{proof}
First, note that \eqref{def_S_n} implies $S_n \subseteq \tilde{S}_n$, since $\omega_n$ vanishes on $S_n$ by Proposition \ref{omega_n}.
For the opposite inclusion, we define, similar to Proposition \ref{interference_prop}, the quantity
\[
\tilde{\Phi}_n := \log\abs{\frac{\tilde{v}_n}{\tilde{v}_{n+1}}} \text{.}
\]
Note that $\smash{\tilde{\Phi}_n}$ satisfies
\begin{align*}
B_n^{-1} \tilde{\Phi}_n &= \sigma - \paren{n^{-1} - \tfrac{2}{3}n^{-2} + C^1_n n^{-3}} + C^2_n n^{-4} \log \abs{c_{n,0}+c_{n,\star}+\omega_n} \\
&\qquad -C ^2_n n^{-4} \log \abs{c_{n+1,0}+c_{n+1,\star}+\omega_{n+1}} \text{,}
\end{align*}
where the constants $B_n$, $C^1_n$, $C^2_n$ are as in the proof of Proposition \ref{interference_prop}.
Proceeding from the above as in the proof of \eqref{graph_S_n}, we then obtain the following properties for $\smash{\tilde{\Phi}_n}$:
\begin{itemize}
\item $\smash{\tilde{\Phi}_n} \leq -K_0 (n+1)^2$ whenever $\sigma \leq (n+1)^{-1} + \frac{1}{6}(n+1)^{-2}$.

\item $\smash{\tilde{\Phi}_n} \geq K_0 n^2$ whenever $\sigma \geq n^{-1} - \frac{1}{6} n^{-2}$.

\item $\partial_\sigma \smash{\tilde{\Phi}_n} \geq K_0 n^4$ whenever $(n+1)^{-1} + \frac{1}{6}(n+1)^{-2} < \sigma < n^{-1} -\frac{1}{6}n^{-2}$.
\end{itemize}
Consequently, for any $y \in \mc{I}$, the quantity $\smash{\tilde{\Phi}_n} (\sigma_0, y)$ can only vanish at a single $\sigma_0 > 0$.
In particular, this yields the inclusion $\tilde{S}_n \subseteq S_n$ and completes the proof of \eqref{S_tilde_n}.

Furthermore, the above bounds on $\tilde{\Phi}_n$ immediately imply
\[
\tilde{\Phi}_n \begin{cases}
  \geq K_0 \cdot \min( n^4(\sigma - \mf{s}_n), n^2) &\quad \sigma \geq \mf{s}_n \text{,} \\
  \leq -K_0 \cdot \min( n^4(\mf{s}_n - \sigma), n^2) &\quad \sigma \leq \mf{s}_n \text{.}
\end{cases}
\]
As a consequence, from the definition of $\tilde{\Phi}_n$, we derive \eqref{proposition_s_n_tilde}:
\begin{align*}
\abs{\tilde{v}_n + \tilde{v}_{n+1}} &\geq \begin{cases}
  \abs{\tilde{v}_n} \abs{1 - e^{-\tilde{\Phi}_n}} &\quad \sigma \geq \mf{s}_n \text{,} \\
  \abs{\tilde{v}_{n+1}} \abs{e^{\tilde{\Phi}_n} - 1} &\quad \sigma \leq \mf{s}_n
\end{cases} \\
&\geq \begin{cases}
  \abs{\tilde{v}_n} \cdot K_0 \min(\sigma - \mf{s}_n, 1) &\quad \sigma \geq \mf{s}_n \text{,} \\
  \abs{\tilde{v}_{n+1}} \cdot K_0 \min(\mf{s}_n - \sigma, 1) &\quad \sigma \leq \mf{s}_n \text{.}
\end{cases} \qedhere
\end{align*}
\end{proof}

\begin{proposition}
For any $n \geq n_0$, the quantity
\begin{equation}
\label{psi_n} \psi_n := \paren{in^{2\alpha} T_1 + T_{2,n}} (c_{n,0}+c_{n,\star}+\omega_n)
\end{equation}
satisfies, for any $N+k \leq I_n^{\star\star}$, the estimate
\begin{align}
\label{estimate_psi_n} \abs{D_n^N \psi_n} &\leq n^\beta n^{2 (\alpha+\beta+k)} n^{(\beta-\alpha) I_n^{\star\star}} \abs{\sigma - \mf{s}_n}^k \text{,} \\
\notag \abs{D_n^N \psi_n} &\leq n^\beta n^{2 (\alpha+\beta+k)} n^{(\beta-\alpha) I_n^{\star\star}} \abs{\sigma - \mf{s}_{n-1}}^k \text{.}
\end{align}
\end{proposition}

\begin{proof}
From the property \eqref{equation_omega_n} of $\omega_n$, the function $\psi_n$ vanishes at order $I_n^{\star\star}$ on $S_n \cup S_{n-1}$.
As a consequence, for any $N+k\leq I_n^{\star\star}$, using \eqref{prop_estimate_sum} and \eqref{estimate_omega_n},
\begin{align*}
\abs{D^N_n \psi_n} &\leq K_{N+k} n^{2k} \abs{\sigma - \mf{s}_n}^k \sup_{ \Omega_n } \abs{D^{N+k}_n \psi_n} \\
&\leq n^\beta n^{2(\alpha+\beta+k)} n^{(\beta-\alpha) I_n^{\star\star}} \abs{\sigma - \mf{s}_n}^k \text{,}
\end{align*}
which is the first part of \eqref{estimate_psi_n}.
The second bound follows from a similar analysis around $S_{n-1}$. 
\end{proof}

We can now finally establish the desired properties of the potential $a$:

\begin{proposition}
$a$ is a smooth function in a neighbourhood of $\sigma = 0$ in $\Omega$.
Furthermore,
\begin{equation}
\label{estimate_a} \lim_{\sigma_0 \searrow 0} \sup_{ \lbrace \sigma = \sigma_0 \rbrace } \abs{\sigma^{-\mu} D^N a} = 0 \text{,} \qquad N, \mu > 0 \text{.}
\end{equation}
\end{proposition}

\begin{proof}
Fix $x \in \Omega$ such that $\sigma (x)$ is sufficiently small---in particular, $x \in \Omega_n$ for some $n > n_0$.
For convenience, we omit $x$ from the notation, though all the estimates below apply just at $x$.
As in the proof of \eqref{u_est}, we can split into two cases, depending on which $\tilde{v}_n (x)$'s are nonzero: 

\vspace{\medskipamount}
\noindent
\textit{1. $x$ lies in the support of only one $\tilde{v}_n$.}
Note $\chi_n = 1$ in this case, hence \eqref{estimate_c_0} and \eqref{estimate_omega_n} imply
\begin{equation}
\label{bound_below} \abs{c_{n,0}+c_{n,\star}+\omega_n} \geq K_0 \text{.}
\end{equation}
Moreover, note from \eqref{T1_T2} that $a (x)$ takes the following form: 
\[
a = \frac{ \mc{P} \tilde{v}_n }{ \tilde{v}_n } = \frac{\psi_n}{c_{n,0}+c_{n,\star}+\omega_n} \text{.}
\]
As a result, using \eqref{estimate_c_0}, \eqref{estimate_c_star}, \eqref{estimate_omega_n}, \eqref{estimate_psi_n}, and \eqref{bound_below}, we can estimate $a$ by
\begin{align}
\label{estimate_a1} \abs{D^N_n a} &\leq K_N \sum_{ \substack{ M+Q=N \text{, } q \leq Q \\ Q_1+\dots+Q_q = Q \text{, } Q_k \neq 0}} \frac{ \abs{D^M_n \psi_n} }{ \abs{c_{n,0}+c_{n,\star}+\omega_n}^{q+1} } \prod_{k=1}^q \abs{D^{Q_k}_n(c_{n,0}+c_{n,\star} + \omega_n)} \\
\notag &\leq n^\beta n^{2(\alpha+\beta)} n^{(\beta-\alpha) I_n^{\star\star}} \\
\notag &\leq n^{-cI_n^{\star\star}} \text{,}
\end{align}
for some constant $c > 0$.
(As usual, we have adjusted $\beta$ and the $K_N$'s as needed.)

\vspace{\medskipamount}
\noindent
\textit{2. $x$ lies only in the supports of $\tilde{v}_n$ and $\tilde{v}_{n+1}$.}
In this case, $a (x)$ takes the form
\[
a = \frac{ \mc{P}\tilde{v}_n + \mc{P}\tilde{v}_{n+1} }{ \tilde{v}_n + \tilde{v}_{n+1} } = \frac{ \zeta_n \psi_n + \zeta_{n+1} \psi_{n+1} }{ \tilde{v}_n + \tilde{v}_{n+1} } \text{,} \qquad \zeta_n := e^{in^{2\alpha}\varphi}e^{f_n} \text{.}
\]
Let us first assume $\sigma \geq \mf{s}_n$ and set $N \leq \smash{\floor{I_n^{\star\star}/10}}$.
In this case, we have
\[
\chi_n = 1 \text{,} \qquad f_{n+1} - f_n \leq K_0 \text{,}
\]
the latter from Proposition \ref{prop_sign_f_n}.
As a consequence, 
\begin{align}
\notag \abs{D^N_n a} &\leq K_N \sum_{ \substack{L+M+Q=N \\ q \leq Q \\ Q_1+\dots+Q_q = Q \\ Q_k \neq 0}} \frac{ \paren{\abs{D^L_n\zeta_n}\abs{D^M_n \psi_n} + \abs{D^L_n\zeta_{n+1}} \abs{D^M_n \psi_{n+1}} } \prod_{k=1}^q \abs{D^{Q_k}_n \paren{ \tilde{v}_n + \tilde{v}_{n+1}}} }{\abs{\tilde{v}_n + \tilde{v}_{n+1}}^{q+1}} \\
\label{estimate_a2_i} &\leq K_N \sum_{ \substack{L+M+Q=N \\ q \leq Q \\ Q_1+\dots+Q_q = Q \\ Q_k \neq 0} } \frac{ \abs{D^M_n \psi_n}+\abs{D^M_n \psi_{n+1}} }{ \min(\abs{\sigma - \mf{s}_n}, 1)^{q+1} } \frac{ \abs{D^L_n \zeta_n} + \abs{D^L_n \zeta_{n+1}} }{ \abs{\tilde{v}_n} } \prod_{k=1}^q \frac{ \abs{D^{Q_k}_n \tilde{v}_n } + \abs{D^{Q_k}_n \tilde{v}_{n+1}} }{\abs{\tilde{v}_n}} \text{.}
\end{align}

We now estimate each factor in \eqref{estimate_a2_i} individually.
From \eqref{estimate_psi_n}, we bound
\begin{equation}
\label{estimate_a2_i1} \frac{ \abs{D^M_n \psi_n}+\abs{D^M_n \psi_{n+1}} }{ \min(\abs{\sigma - \mf{s}_n}, 1)^{q+1} } \leq n^{ \beta + 2(\alpha + \beta + q + 1) + (\beta-\alpha) I_n^{\star\star}} + (n+1)^{ \beta + 2(\alpha + \beta + q + 1) + (\beta-\alpha) I_n^{\star\star}} \text{.}
\end{equation}
For the second factor, one obtains, from Proposition \ref{prop_sign_f_n} and \eqref{bound_below}: 
\begin{align}
\label{estimate_a2_i2} \frac{ \abs{D^L_n \zeta_n} + \abs{D^L_n \zeta_{n+1}} }{ \abs{\tilde{v}_n} } &\leq K_N \frac{ n^{2L\alpha}e^{f_n} + (n+1)^{2L\alpha}e^{f_{n+1}} }{ e^{f_n} \abs{c_{n,0}+c_{n,\star}+\omega_n} } \\
\notag &\leq (n+1)^{\beta + 2L\alpha} \text{.}
\end{align}
Similarly, from Proposition \ref{prop_sign_f_n}, \eqref{v_n_tilde}, and \eqref{bound_below}, we estimate
\begin{align}
\label{estimate_a2_i3} \frac{ \abs{D^{Q_k}_n \tilde{v}_n} + \abs{D^{Q_k}_n \tilde{v}_{n+1}} }{ \abs{\tilde{v}_n} } &\leq K_N \frac{ n^{2Q_k\alpha} e^{f_n} + (n+1)^{2 Q_K \alpha} e^{f_{n+1}} }{ e^{f_n} \abs{c_{n,0}+ c_{n,\star}+\omega_n} } \\
\notag &\leq (n+1)^{\beta + 2Q_k\alpha} \text{.}
\end{align}
Combining \eqref{estimate_a2_i}--\eqref{estimate_a2_i3}, we finally obtain (adjusting constants as needed)
\begin{equation}
\label{estimate_a2} \abs{D^N_n a} \leq (n+1)^{3\beta + 3N\alpha} n^{(\beta-\alpha) I_n^{\star\star}} \leq \paren{\tfrac{n}{2}}^{-cI_n^{\star\star}} \text{,} \qquad c > 0 \text{.}
\end{equation}

Furthermore, in the remaining case $\sigma \leq \mf{s}_n$, we can bound
\[
\abs{D^N_n a} \leq K_N \sum_{ \substack{L+M+Q=N \\ q \leq Q \\ Q_1+\dots+Q_q = Q \\ Q_k \neq 0} } \frac{ \abs{D^M_n \psi_n}+\abs{D^M_n \psi_{n+1}} }{ \min(\abs{\sigma - \mf{s}_n}, 1)^{q+1} } \frac{ \abs{D^L_n \zeta_n} + \abs{D^L_n \zeta_{n+1}} }{ \abs{\tilde{v}_{n+1}} } \prod_{k=1}^q \frac{ \abs{D^{Q_k}_n \tilde{v}_n } + \abs{D^{Q_k}_n \tilde{v}_{n+1}} }{\abs{\tilde{v}_{n+1}}} \text{,}
\]
from which a similar analysis as above also yields \eqref{estimate_a2}.

\vspace{\medskipamount}
\noindent
Combining all the cases, we see that $a$ is indeed everywhere defined near $\sigma = 0$.
Moreover, the desired vanishing \eqref{estimate_a} follows from the estimates \eqref{estimate_a1} and \eqref{estimate_a2}.
\end{proof}

In particular, $u$ and $a$ vanish to infinite order at $\sigma = 0$.
Moreover, since $u$ only vanishes at the $S_n$'s (at least near $\sigma = 0$), the support of $u$ contains a neighbourhood $\Omega \cap \{ 0 < \sigma < \sigma' \}$.
Lastly, to extend the support of $u$, and the function $a$ itself, to all of $\bar{\Omega} \cap \{ \sigma \geq 0 \}$, we can simply arbitrarily modify $u$ (and hence $a$) far away from $\sigma = 0$ to any non-vanishing function.

\section{Asymptotically Anti-de Sitter Spacetimes} \label{sec.aads}

In this section, we apply Theorem \ref{Theorem} to \emph{asymptotically Anti-de Sitter} (\emph{aAdS}) spacetimes, which arise as solutions of the Einstein field equations with a negative cosmological constant and are of interest in both relativity and holography.
We will construct counterexamples to unique continuation for Klein--Gordon equations from appropriate regions of the conformal boundary.

To make our discussions more concrete, we will first detail our constructions on two prototypical spacetimes:\ \emph{planar AdS} (Section \ref{sec:planar}) and \emph{AdS} (Section \ref{sec:pure}).
Then, in Section \ref{sec:aads_general}, we briefly discuss how these counterexamples can be adapted to general aAdS spacetimes, and we connect this to the positive unique continuation results of \cite{Shao22, hol_shao:uc_ads, hol_shao:uc_ads_ns, McGill20}.

\subsection{Planar AdS} \label{sec:planar}

Planar AdS spacetime can be represented as the manifold $(\mc{M}_{plan}, g_{plan})$, with
\begin{equation}
\label{aads_planar} \mc{M}_{plan} = (0, \infty)_r \times \R^d_{ (t, \bar{x}) } \text{,} \qquad g_{plan} := r^{-2} dr^2 + r^2 \eta \text{,}
\end{equation}
where $\eta = -dt^2 + d\bar{x}^2$ is the Minkowski metric in $d$ dimensions.
In particular, $(\mc{M}_{plan}, g_{plan})$ is a solution of the Einstein-vacuum equations (EVE), with cosmological constant $\Lambda = -d(d-1)/2$.

Taking $\rho := r^{-1}$, one sees that planar AdS is conformally isometric to the Minkowski half-space: 
\begin{equation}
\label{aads_planar_conformal} \bar{\mc{M}}_{plan} = (0, \infty)_\rho \times \R^d_{(t, \bar{x})} \text{,} \qquad \bar{g}_{plan} := d \rho^2 + \eta = \rho^2 g_{plan} \text{.}
\end{equation}
We refer to the boundary $\partial \bar{\mc{M}}_{plan} := \{ \rho = 0 \}$ of \eqref{aads_planar_conformal} as the \emph{conformal boundary} of planar AdS.
Note that $\partial \bar{\mc{M}}_{plan}$ inherits a natural Lorentzian structure as $( \R^d, \eta )$.

Consider, for instance, the Klein--Gordon operator on planar AdS:
\begin{equation}
\label{kg_planar} \mc{P}_\mu := \Box_{g_{plan}} + \mu \text{,} \qquad \mu \in \R \text{.}
\end{equation}
Observe that \eqref{kg_planar} is conformally equivalent to the singular wave operator 
\begin{equation}
\label{kg_planar_conformal} \bar{\mc{P}}_\mu := \Box_{\bar{g}_{plan}} + \rho^{-2} \big( \mu - \tfrac{d^2 - 1}{4} \big) + V \text{,}
\end{equation}
with $V$ being a specific smooth and bounded function on $\bar{\mc{M}}_{plan}$.
(More specifically, if $\phi$ solves $\mc{P}_\mu \phi = 0$, then \smash{$\bar{\phi} := \rho^{-(d-1)/2} \phi$} solves \smash{$\bar{\mc{P}}_\mu \bar{\phi} = 0$}.)
In particular, if $\mu \neq (d^2 - 1)/4$, then the operator \eqref{kg_planar_conformal} contains a critically singular potential as in \eqref{general_op}.
As a result, to construct a counterexample to unique continuation for \eqref{kg_planar} from the conformal boundary of planar AdS, it suffices to construct a counterexample to unique continuation for \eqref{kg_planar_conformal} from $\rho = 0$.

\begin{figure}
\centering
\begin{subfigure}{.45\textwidth}
  \centering
  \includegraphics[scale=0.85]{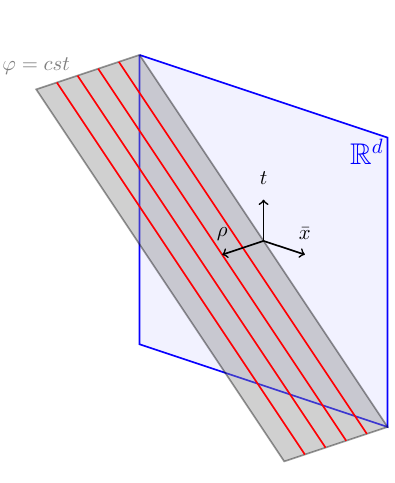}
  \caption{}
\end{subfigure}
\begin{subfigure}{.45\textwidth}
  \centering
  \includegraphics[scale=0.85]{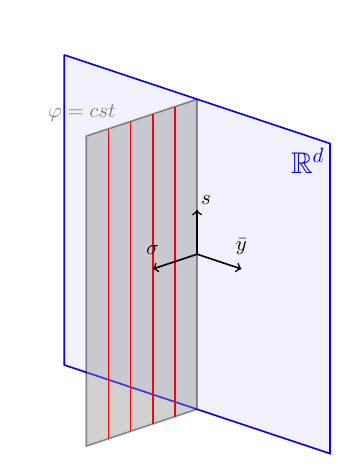}
  \caption{}
\end{subfigure}

\caption{These illustrations represent the change of coordinates from $(\rho, t, x)$ in (A) to $(\sigma, s, y)$ in (B).
In (B), the gradient of $\varphi$ is aligned with the coordinate vector field $\partial_s$.
Null geodesics in both graphics are drawn in \textcolor{red}{red}.}
\label{fig:planar}
\end{figure}

In the following, we will transform the above setting such that Theorem \ref{Theorem} can be applied to $(\bar{\mc{M}}_{plan}, \bar{g}_{plan})$.
First, fix a (Euclidean) unit vector $\bar{k} \in \Sph^{d-2} \subseteq \R^{d-1}$, and let
\begin{equation}
\label{eikonal_planar} \varphi := \tfrac{1}{2} ( \bar{k} \cdot \bar{x} - t ) \in C^\infty ( \bar{\mc{M}}_{plan} ) \text{,} \qquad \sigma := \rho \in C^\infty ( \bar{\mc{M}}_{plan} ) \text{,}
\end{equation}
where ``$\cdot$" is the Euclidean dot product.
Note $\varphi$ and $\sigma$ satisfy \eqref{boundedness_metric} and the first part of \eqref{conditions_s}:
\[
\bar{g}_{plan}^{-1} ( d \varphi, d \varphi ) = 0 \text{,} \qquad \bar{g}_{plan}^{-1} ( d \sigma, d \sigma ) \equiv 1 \text{,} \qquad 2 \, \grad_{ \bar{g}_{plan} } \varphi = \partial_t + \bar{k} \cdot \nabla_{ \bar{x} } \text{.}
\]
In particular, the level sets of $\varphi$ are null hyperplanes in $\bar{\mc{M}}_{plan}$ whose orientations are given by $\bar{k}$.

Furthermore, we make the following change of coordinates on $\R^d$:
\begin{equation}
\label{sy_planar} s := t \text{,} \qquad \bar{y} := \bar{x} - t \bar{k} \text{.}
\end{equation}
Clearly, $( \sigma, s, \bar{y} )$ defines a global coordinate system on $\bar{\mc{M}}_{plan}$, and one can immediately compute
\[
2 \, \grad_{ \bar{g}_{plan} } \varphi = \partial_s \text{,}
\]
so that the second part of \eqref{conditions_s} is also satisfied.
Now, as the above is well-defined on $s_- < s < s_+$ for any $s_- < s_+$, and recalling \eqref{sy_planar}, then we can apply Theorem \ref{Theorem} to obtain a counterexample to unique continuation for $\bar{\mc{P}}_\mu$ from $\rho = 0$ in any a slab $\Omega := \lbrace \sigma < \sigma_0 \rbrace \cap \lbrace s_- < t < s_+ \rbrace$.

Finally, pulling the above constructions back to $( \mc{M}_{plan}, g_{plan} )$, we obtain:

\begin{corollary} \label{corollary_planar_1}
Let $(\mc{M}_{plan}, g_{plan})$ be as in \eqref{aads_planar}, and fix $\mu \in \R$.
Then, for any $t_- < t_+$, there exist $r_0 > 0$ and $u, a \in C^\infty ( \Omega )$, with $\Omega := \{ r > r_0 \} \cap \{ t_- < t < t_+ \}$, such that:
\begin{itemize}
\item $u$ is supported on all of $\Omega$.

\item $u$ and $a$ both vanish to infinite order as $r \nearrow +\infty$.

\item $u$ satisfies the following Klein--Gordon equation on $\Omega$:
\begin{equation}
\label{corollary_planar_kg} ( \Box_{ g_{plan} } + \mu ) u = au \text{.} 
\end{equation}
\end{itemize}
\end{corollary}

\begin{remark}
Notice the timespan $t_+ - t_-$ in Corollary \ref{corollary_planar_1} can be arbitrarily large.
This reinforces the fact that the unique continuation results of \cite{Shao22, hol_shao:uc_ads, hol_shao:uc_ads_ns, McGill20} fail on planar AdS from regions on the conformal boundary spanning an arbitrarily large amount of time.
\end{remark}

\begin{remark}
In fact, $r_0$ in Corollary \ref{corollary_planar_1} can be arbitrarily chosen, since $u$ and $a$ can be smoothly extended further inward by taking $u$ to be any arbitrary nonvanishing function.
\end{remark}

\begin{remark}
Note that the special case of \eqref{corollary_planar_kg} with the conformal mass, $\mu := ( d^2 - 1 )/4$, can also be treated using \cite{AB}.
However, since the construction of \cite{AB} is purely local in nature, one would still need Corollary \ref{corollary_planar_1} to produce counterexamples covering an arbitrarily large timespan.
\end{remark}

One feature, which is a consequence of the construction in Theorem \ref{Theorem}, is that the counterexample $u$ obtained from Corollary \ref{corollary_planar_1} is supported on all of $\Omega$.
This could be undesirable, as this $u$ fails to be in $L^2 ( \Omega )$.
One attempt to circumvent this is to restrict $u$ in Corollary \ref{corollary_planar_1} to a bounded subset of $\bar{y}$-values and then consider $u_\ast := \chi u$ for an appropriate cutoff function $\chi$.
However, this $u_\ast$ would, even worse, fail to satisfy an equation of the form \eqref{corollary_planar_kg}, as the corresponding error potential $a_\ast := u_\ast^{-1} ( \Box_{ g_{plan} } + \mu ) u_\ast$ would not be smoothly zero-extendible to $\Omega$.

Next, we demonstrate a trick to overcome this limitation, producing counterexamples on all of $\Omega$ whose supports are nonetheless localised to a bounded sector of null geodesics.
For this, we return to $( \sigma, s, \bar{y} )$-coordinates, and we fix bounded open subsets $B_i \Subset B_o \Subset \R^{d-1}$.
We also fix $0 < \sigma_1 < \sigma_0$ and a smooth function $\Psi: \R^{d-1} \rightarrow [ 0, \sigma_1 ]$ satisfying
\[
\Psi ( \bar{z} ) := \begin{cases}
  0 & \bar{z} \in B_i \text{,} \\
  \sigma_1 & \bar{z} \not\in B_o \text{,}
\end{cases}
\]
and we deform $\sigma$ by setting
\begin{equation}
\label{bump_function} \tilde{\sigma} := \sigma - \Psi ( \bar{y} ) \text{.}
\end{equation}
See Figure \ref{fig:trick_ads} for an illustration of a level set of $\tilde{\sigma}$; observe that level sets of $\sigma$ and $\tilde{\sigma}$ coincide when $\bar{y} \in B_i$, while the level sets of $\tilde{\sigma}$ protrude away from the conformal boundary when $\bar{y} \not\in B_i$.

The crucial observation, then, is that the level sets of $\tilde{\sigma}$ are ruled by the same family of geodesics as the level sets of $\sigma$---namely, those generated by $\varphi$.
As a result, we can now apply Theorem \ref{Theorem} in $( \tilde{\sigma}, s, \bar{y} )$-coordinates instead, which yields a counterexample $\bar{u}$ to unique continuation for $\bar{\mc{P}}_\mu$ \emph{from $\tilde{\sigma} = 0$ that is supported on all of $\tilde{\sigma} > 0$}.
(Moreover, note the potential becomes less singular than before when $\sigma \neq \tilde{\sigma}$.)
Furthermore, since $\bar{u}$ vanishes to infinite order at $\tilde{\sigma} = 0$, it can be smoothly zero-extended to $\tilde{\sigma} < 0$.
By restricting $\bar{u}$ to $\smash{\tilde{\Omega}} := ( 0, \sigma_1 ) \times \R^d$, we obtain a counterexample to unique continuation for $\bar{\mc{P}}_\mu$ whose support contains the sector $\bar{y} \in B_i$ but also lies within the sector $\bar{y} \in \bar{B}_o$.
Pulling this back to $( \mc{M}_{plan}, g_{plan} )$ results in the following:

\begin{figure}
\centering
\begin{subfigure}{.55\textwidth}
  \centering
  \includegraphics[scale=0.95]{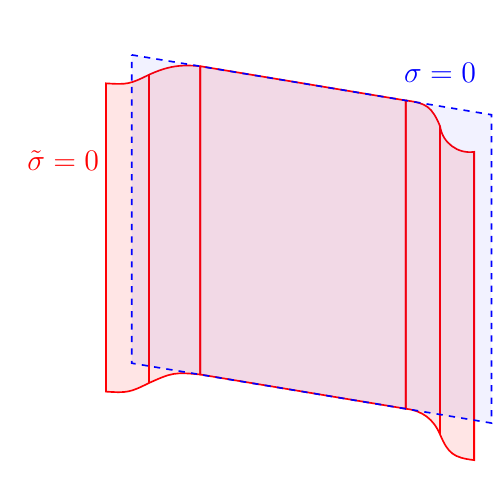}
  \caption{}
  \label{fig:trick_ads}
\end{subfigure}
\begin{subfigure}{.35\textwidth}
  \centering
  \includegraphics{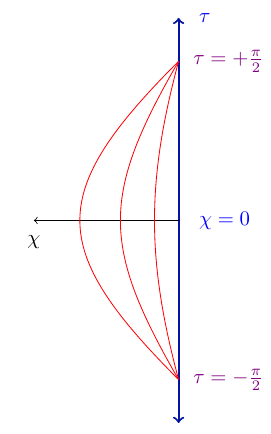}
  \caption{}
  \label{fig:sub1}
\end{subfigure}

\caption{(A) illustrates both the conformal boundary \textcolor{blue}{$\sigma = 0$} and the deformed boundary \textcolor{red}{$\tilde{\sigma} = 0$} used to construct the localised counterexamples of Corollary \ref{corollary_planar_2}.
(B) illustrates a family of null geodesics in pure AdS spacetime \eqref{aads_pure_conf}, projected to the $\tau$-$\chi$-plane; in particular, these geodesics (in \textcolor{red}{red}) start from the conformal boundary, remain near the boundary, and return to the boundary after time $\pi$.}
\end{figure}

\begin{corollary} \label{corollary_planar_2}
Assume the setting and notations of Corollary \ref{corollary_planar_1}, and fix bounded open subsets $B_i \Subset B_o \Subset \R^{d-1}$.
Then, there exist $r_0 > 0$ and $u, a \in C^\infty ( \Omega )$ such that:
\begin{itemize}
\item The support of $u$ is nontrivial and contained within a bounded sector of null geodesics:
\begin{equation}
\label{corollary_planar_support} \Omega \cap \{ \bar{x} - t \bar{k} \in B_i \} \subseteq \operatorname{supp} u \subseteq \Omega \cap \{ \bar{x} - t \bar{k} \in \bar{B}_o \} \text{.}
\end{equation}

\item $u$ and $a$ both vanish to infinite order as $r \nearrow +\infty$.

\item $u$ satisfies the Klein--Gordon equation \eqref{corollary_planar_kg} on $\Omega$.
\end{itemize}
\end{corollary}

\begin{remark}
Note that Corollaries \ref{corollary_planar_1} and \ref{corollary_planar_2} still hold if the Klein--Gordon operator is replaced by a more general $\Box_{ g_{plan} } + \mu + W$ for a sufficiently nice potential $W$.
\end{remark}

\subsection{Pure AdS} \label{sec:pure}

AdS spacetime is the maximally symmetric solution of the EVE with a negative cosmological constant.
Outside the origin ($r = 0$), it can be represented as $(\mc{M}_{AdS}, g_{AdS})$, where
\begin{equation}
\label{aads_pure} \mc{M}_{AdS} := \R_\tau \times ( 0, +\infty )_r \times \Sph^{d-1}_\omega \text{,} \qquad g_{AdS} := -(1+r^2) d \tau^2 + (1+r^2)^{-1} dr^2 + r^2 \, \slashed{g}(\omega)\text{,}
\end{equation}
where $\slashed{g}$ denotes the unit round metric on $\Sph^{d-1}$.
We now use Theorem \ref{Theorem} to construct counterexamples to unique continuation from the AdS conformal boundary for the operator
\begin{equation}
\label{kg_pure} \mc{P}_\mu := \Box_{g_{AdS}} + \mu \text{,} \qquad \mu \in \R \text{.}
\end{equation}

Applying the transformation $\chi := \pi/2 - \arctan r$, we can also reformulate AdS spacetime as
\begin{equation}
\label{aads_pure_conf} \mc{M}_{AdS} := \R_\tau \times \big( 0, \tfrac{\pi}{2} \big)_\chi \times \Sph^{d-1}_\omega \text{,} \qquad g_{AdS} := \tfrac{1}{ \sin^2 \chi } [ -d\tau^2 + d\chi^2 + \cos^2 \chi \cdot \slashed{g}(\omega) ] \text{.}
\end{equation}
In particular, $( \mc{M}_{AdS}, g_{AdS} )$ is conformally isometric to half of the Einstein cylinder $\R \times \Sph^d$, and the conformal boundary of AdS spacetime can then be formally realised as the hypersurface $\chi = 0$, which inherits the Lorentzian structure of $( \R \times \Sph^{d-1}, -d \tau^2 + \slashed{g} )$.

Particular features of AdS geometry---most notably, that its conformal boundary has positively curved and compact cross-sections---make the construction of counterexamples to unique continuation more delicate than on planar AdS.
In particular, null geodesics near the conformal boundary must return to the boundary within a fixed amount of time; see Figure \ref{fig:sub1}.
(This can be understood from the observation that the null geodesics spatially project to great circles on $\Sph^d$, so that any null geodesic starting from the conformal boundary returns to the boundary after time $\pi$.)
This limits the timespan along which counterexamples can be constructed.

More specifically, \cite{hol_shao:uc_ads} showed that unique continuation for \eqref{kg_pure} holds from a slab $\{ \tau_- < \tau < \tau_+ \}$ of the conformal boundary with timespan $\tau_+ - \tau_- > \pi$.
(As above, here we identify the conformal boundary with $\R_\tau \times \Sph^{d-1}_\omega$.)
We next demonstrate that \emph{the above is in fact optimal}, \emph{by constructing counterexamples from boundary slabs with timespan $\tau_+ - \tau_- < \pi$}.

We also remark that the compactness of AdS boundary cross-sections provides yet another difficulty to our construction.
Whereas in planar AdS, one can globally choose a direction of propagation for the crucial null geodesics (namely, the constant vector $\bar{k}$), this is no longer possible in general on the spherical background of AdS.
In other words, \emph{an appropriate eikonal function $\varphi$ cannot be found in a neighbourhood of an entire slab $\{ \tau_- < \tau < \tau_+ \}$ of the conformal boundary}.

Consequently, we must also apply the trick of deforming $\sigma$ in Corollary \ref{corollary_planar_2} in order to localise our construction, so that our counterexamples can be smoothly defined all of $\{ \tau_- < \tau < \tau_+ \}$.
In particular, this trick now becomes an essential part of the construction in AdS spacetime.

\medskip
\noindent
\textit{Construction of counterexamples.}
One strategy for generating the desired counterexamples is to apply Theorem \ref{Theorem} directly with respect to the conformally related $\bar{g}_{AdS} := \sin^2 \chi \cdot g_{AdS}$.
However, here we opt for a different method; we instead exploit the relation between AdS and planar AdS spacetimes in order to take advantage of our constructions from Corollary \ref{corollary_planar_2}.

For this, we first fix $0 < \epsilon < \frac{\pi}{2}$, and we consider the region of AdS spacetime given by
\begin{equation}
\label{pure_planar_region} \mc{M}_{P, \epsilon} := \big\lbrace |\tau| \leq \tfrac{\pi}{2} - \epsilon \text{, } \omega^d < 0 \big\rbrace \subseteq \mc{M}_{AdS} \text{,}
\end{equation}
where we view $\omega \in \Sph^{d-1} \subseteq \R^d$ as a $d$-vector in Cartesian coordinates.
Note $\mc{M}_{P, \epsilon}$ covers precisely half of the slab $| \tau | \leq \frac{\pi}{2}$ of $\bar{\mc{M}}_{AdS}$, by restricting to half of the $\Sph^{d-1}$-component.

Next, since $\mc{M}_{P, \epsilon}$ is contained in the Poincar\'e patch of AdS spacetime, $(\mc{M}_{P, \epsilon}, g_{AdS})$ isometrically embeds into a portion of the planar AdS spacetime \eqref{aads_planar}.
More specifically, by adapting \cite{Bayona_2007, Aharony_2000}, for instance, and applying the coordinate transformation
\begin{equation}
\label{aads_transfo_pure} t := \frac{\sin \tau}{\cos\tau - \cos\chi \, \omega^d} \text{,} \qquad \rho := \frac{\sin\chi}{\cos\tau - \cos\chi \, \omega^d} \text{,} \qquad \bar{x}^j := \frac{\cos\chi\, \omega^j}{\cos\tau - \cos\chi \omega^d} \text{,}
\end{equation}
for any $1 \leq j < d$, we then see that $g_{AdS}$ on $\mc{M}_{P, \epsilon}$ takes the form
\begin{equation}
\label{aads_transfo_plan} g_{AdS} = \rho^{-2} ( d \rho^2 - dt^2 + d \bar{x}^2 ) = g_{plan} \text{.}
\end{equation}
Note the coordinates \eqref{aads_transfo_pure} are smooth and bounded on $\mc{M}_{P, \epsilon}$:
\begin{equation}
\label{aads_transfo_bound} \sup_{ \mc{M}_{P,\epsilon} } ( |t| + \rho + | \bar{x} | ) \leq c \epsilon^{-1} \text{,} \qquad c > 0 \text{.}
\end{equation}
Furthermore, observe that $\rho = 0$ here corresponds to the conformal boundary $\chi = 0$.

Fix now a unit vector $\bar{k} \in \Sph^{d-2} \subseteq \R^{d-1}$.
Then, Corollary \ref{corollary_planar_2} yields a counterexample
\[
u \in C^\infty ( \mc{M}_{plan} \cap \{ \rho < \rho_0 \} \cap \{ |t| \leq c \epsilon^{-1} \} )
\]
(with a corresponding $a$) such that \eqref{corollary_planar_kg} holds, $u$ vanishes to infinite order as $\rho \searrow 0$, and
\begin{align}
\label{pure_cex_supp} \operatorname{supp} u \subseteq ( \mc{M}_{plan} \cap \{ \rho < \rho_0 \} \cap \{ |t| \leq c \epsilon^{-1} \} ) \cap \{ | \bar{x} - t \bar{k} | \leq \delta \epsilon \} \text{,} \\
\notag \operatorname{supp} u \supseteq ( \mc{M}_{plan} \cap \{ \rho < \rho_0 \} \cap \{ |t| \leq c \epsilon^{-1} \} ) \cap \big\{ | \bar{x} - t \bar{k} | \leq \tfrac{\delta}{2} \epsilon \big\} \text{,}
\end{align}
for some sufficiently small constant $0 < \delta \ll 1$.

\begin{figure}
\centering
\begin{subfigure}{.48\textwidth}
  \centering
  \includegraphics{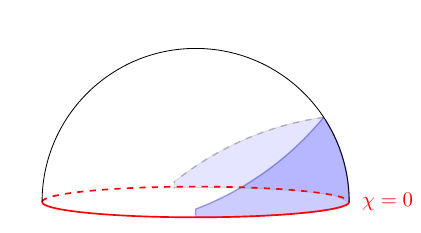}
  \caption{}
  \label{fig:pure1}
\end{subfigure}
\begin{subfigure}{.48\textwidth}
  \centering
  \includegraphics[scale=0.95]{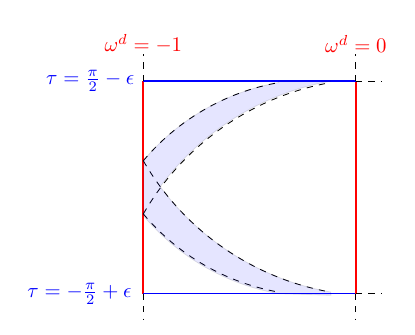}
  \caption{}
  \label{fig:pure2}
\end{subfigure}
\caption{(A) illustrates, at a fixed time $\tau$, the portion $\rho < \rho_0$ of $\mc{M}_{P, \epsilon}$ (shaded in \textcolor{blue}{blue}), on which the restricted counterexample $u_\ast$ is defined.
(B) shows the support of $u_\ast$ (shaded in \textcolor{blue}{blue}), projected to the $\tau$-$\omega^d$-plane; since the support lies away from $\textcolor{red}{\omega^d = 0}$, then $u_\ast$ can be zero-extended to $\omega^d \geq 0$.}
\end{figure}

Let $u_\ast \in C^\infty ( \mc{M}_{P, \epsilon} )$ denote the restriction of $u$ to $\mc{M}_{P, \epsilon}$; see Figure \ref{fig:pure1} for a rough depiction of its domain.
The crucial property needed for $u_\ast$ is the following (see also Figure \ref{fig:pure2}):

\begin{lemma} \label{pure_ustar}
If $\rho_0$ is small enough, then $u_\ast$ vanishes in a neighbourhood of $\{ |\tau| \leq \tfrac{\pi}{2} - \epsilon \text{, } \omega^d = 0 \}$.
\end{lemma}

\begin{proof}
From \eqref{pure_cex_supp}, it suffices to show that when $| \tau | \leq \frac{\pi}{2} - \epsilon$, the region $\{ | \bar{x} - t \bar{k} | \leq \delta \epsilon \}$ lies within $\{ \omega^d < 0 \}$.
For this, we note from \eqref{aads_transfo_pure} that $| \tau | \leq \frac{\pi}{2} - \epsilon$ and $\bar{x} - t \bar{k} \leq \delta \epsilon$ imply
\begin{equation}
\label{eql.omegad_0} |t| \leq c_1 \epsilon^{-1} \text{,} \qquad \big| | \bar{x} |^2 - t^2 \big| = | \bar{x} - t \bar{k} | | \bar{x} - t \bar{k} + 2 t \bar{k} | \leq c_2 \delta \text{,}
\end{equation}
for some universal constants $c_1, c_2 > 0$.
Now, from \cite[Section III]{Bayona_2007}, we have the relation
\[
\omega^d = \frac{ \rho^2 - 1 + | \bar{x} |^2 - t^2 }{ \sqrt{ ( \rho^2 + 1 + | \bar{x} |^2 - t^2 )^2 + 4 t^2 - 4 \rho^2 } } \text{.}
\]
The above, along with the estimates \eqref{eql.omegad_0} and the fact that both $\rho_0$ and $\delta$ are sufficiently small, yields that indeed $\omega^d < 0$, which completes the proof of the lemma.
\end{proof}

By Lemma \ref{pure_ustar}, we can smoothly extend $u_\ast$ by zero to the entire slab $\{ |\tau| \leq \frac{\pi}{2} - \epsilon \} \subseteq \mc{M}_{AdS}$.
Since $g_{AdS} = g_{plan}$ on $\mc{M}_{P, \epsilon}$, then $u_\ast$ also serves as a counterexample to unique continuation from the conformal boundary for \eqref{kg_pure}.
Finally, since $\epsilon$ was arbitrarily chosen, and since AdS spacetime is time-symmetric, we can reformulate our result in terms of \eqref{aads_pure} as follows:

\begin{corollary} \label{aads_corollary_pure}
Let $(\mc{M}_{AdS}, g_{AdS})$ be as in \eqref{aads_pure}, and fix $\mu \in \R$.
Then, for any $\tau_- < \tau_+$ satisfying $\tau_+ - \tau_- < \pi$, there exist $r_0 > 0$ and $u, a \in C^\infty ( \Omega )$, with $\Omega := \{ r > r_0 \} \cap \{ \tau_- < \tau < \tau_+ \}$, such that:
\begin{itemize}
\item The support of $u$ is nontrivial and satisfies, for some $R \simeq \pi - ( \tau_+ - \tau_- )$:
\begin{equation}
\label{corollary_pure_support} \Omega \cap \big\{ | \bar{x} - t \bar{k} | \leq \tfrac{R}{2} \big\} \subseteq \operatorname{supp} u \subseteq \Omega \cap \{ | \bar{x} - t \bar{k} | \leq R \} \text{.}
\end{equation}

\item $u$ and $a$ both vanish to infinite order as $r \nearrow +\infty$.

\item $u$ satisfies the following Klein--Gordon equation on $\Omega$:
\begin{equation}
\label{corollary_pure_kg} ( \Box_{ g_{AdS} } + \mu ) u = au \text{.} 
\end{equation}
\end{itemize}
\end{corollary}

Corollary \ref{aads_corollary_pure} completes our understanding of the unique continuation properties of \eqref{kg_pure}, up to zero-order perturbations.
In particular, it shows that unique continuation fails when $\tau_+ - \tau_- < \pi$, while the results of \cite{Shao22, McGill20, hol_shao:uc_ads, hol_shao:uc_ads_ns} show that unique continuation holds when $\tau_+ - \tau_- > \pi$.

\begin{remark}
Similar to the planar AdS case, Corollary \ref{aads_corollary_pure} still holds if the Klein--Gordon operator is replaced by a more general $\Box_{ g_{AdS} } + \mu + W$ for a sufficiently nice potential $W$.
Furthermore, $r_0$ in Corollary \ref{aads_corollary_pure} can once again be arbitrarily chosen.
\end{remark}

\begin{remark}
One advantage of our approach is that we can provide an explicit formula for the eikonal function $\varphi$ used to construct our counterexamples in Corollary \ref{aads_corollary_pure}:
\[
\varphi = \tfrac{1}{2} ( \bar{k} \cdot \bar{x} - \bar{t} ) = \frac{ \cos \chi \, \, \bar{k} \cdot ( \omega^1, \dots, \omega^{d-1} ) - \sin \tau }{ 2 ( \cos\tau - \cos\chi \, \omega^d ) } \text{.}
\]
\end{remark}

\subsection{General aAdS spacetimes}\label{sec:aads_general}

We now briefly sketch how Theorem \ref{Theorem} can be applied for general aAdS spacetimes.
As details of the analysis would vary depending on the specific properties of the spacetime under consideration, here we opt for conciseness by keeping this discussion informal, and by having this act as a proof of concept, from which a more detailed analysis can be performed for more specific situations.
Let us begin by describing the aAdS geometries we consider:

\begin{assumption}[aAdS spacetime] \label{fg_aads}
Let $( \mc{M}, g )$ be a Lorentzian manifold of the form
\begin{equation}
\label{fg_aads_mfld} \mc{M} := ( 0, \rho_0 )_\rho \times \mc{I} \text{,} \qquad g := \rho^{-2} [ d \rho^2 + \ms{g} ( \rho ) ] \text{,}
\end{equation}
where $\rho_0 > 0$; where $\mc{I}$ a $d$-dimensional manifold, with $d \geq 2$; and where $\ms{g}$ denotes a $\rho$-parametrised family of smooth Lorentzian metrics on $\mc{I}$ that also has the following expansion from $\rho = 0$,
\begin{equation}
\label{fg_aads_exp} \ms{g} ( \rho ) = \mf{g}^{(0)} + \rho^2 \mf{g}^{(2)} + \mc{O} ( \rho^3 ) \text{,}
\end{equation}
with $\mf{g}^{(0)}$ being a Lorentzian metric on $\mc{I}$.
\end{assumption}

Assumption \ref{fg_aads} is an informal statement of the notion of \emph{strongly FG-aAdS segment}, which was more precisely defined and used in \cite{Shao22, McGill20}.
In particular, this is the setting for which the most general unique continuation results for Klein--Gordon equations hold; see \cite[Corollary 5.12]{Shao22}.
The specific form of the metric $g$ in \eqref{fg_aads_mfld}---which in practice arises from a specific choice of a coordinate $\rho$---is known as a \emph{Fefferman--Graham} (abbreviated \emph{FG}) \emph{gauge}.

\begin{remark}
Note that the representation \eqref{aads_planar_conformal} yields that planar AdS is aAdS in the sense of Assumption \ref{fg_aads}.
Pure AdS \eqref{aads_pure} can also be shown to be aAdS via the coordinate change
\[
4 r := \rho^{-1} ( 2 + \rho ) ( 2 - \rho ) \text{.}
\]
\end{remark}

Once again, the aim is to construct counterexamples to unique continuation for the operator
\begin{equation}
\label{kg_aads} \mc{P}_\mu := \Box_g + \mu \text{,} \qquad \mu \in \R
\end{equation}
from the \emph{conformal boundary} ``$\rho = 0$", which can formally manifested as the Lorentzian manifold $( \mc{I}, \mf{g}^{(0)} )$.
As before, it suffices to construct counterexamples for a conformally related operator 
\begin{equation}
\label{aads_fg_conformal} \bar{\mc{P}}_\mu := \Box_{\bar{g}} + \rho^{-2} \big( \mu - \tfrac{d^2 - 1}{4} \big) + V \text{,} \qquad \bar{g} := \rho^2 g = d \rho^2 + \ms{g} ( \rho ) \text{.}
\end{equation}
For this, the key step is to again define the appropriate objects and coordinates (most importantly $\varphi$, $\sigma$, and $s$) so that Theorem \ref{Theorem} can be applied.
In fact, one can view this process as a generalization of the argument for planar AdS detailed in Section \ref{sec:planar}.

\vspace{\medskipamount}
\noindent
\textit{Construction of $\varphi$ and $\sigma$.}
We begin by fixing a spacelike (with respect to $\mf{g}^{(0)}$) hypersurface $\mc{H} \subseteq \mc{I}$, which one can view as a section of the conformal boundary.
By shrinking $\mc{H}$ if necessary, we can also assume that all of $\mc{H}$ is covered by a single bounded coordinate system $\bar{x} := ( x^1, \dots, x^{d-1} )$.
In particular, we can think of $\mc{H}$ as being foliated by level sets of $x^1$.

Next, we extend $\mc{H}$ along $\rho$ to obtain a spacelike (at least for small $\rho_0$) hypersurface in $\mc{M}$:
\begin{equation}
\label{aads_Sigma} \Sigma := ( 0, \rho_0 ) \times \mc{H} \subseteq \mc{M} \text{.}
\end{equation}
Note in particular that $( \rho, \bar{x} )$ then gives a coordinate system on all of $\Sigma$.
In addition:
\begin{itemize}
\item Let $T$ be the unit timelike normal to $\Sigma$ in $\mc{M}$, with respect to $\bar{g}$.

\item Let $E$ be the unit (spacelike) normal to the level sets of $x^1$ in $\Sigma$, also with respect to $\bar{g}$.
\end{itemize}
As a result, the vector field on $\Sigma$ given by
\begin{equation}
\label{aads_N} N := T + E
\end{equation}
is everywhere $\bar{g}$-null and $\bar{g}$-normal to the level sets of $x^1$ in $\Sigma$.

We now consider the family $( \Lambda_p )_{p \in \Sigma}$ of affinely parametrised, maximal, null geodesics with
\begin{equation}
\label{aads_Lambda} \Lambda_p (0) = p \text{,} \qquad \Lambda'_p (0) = N |_p \text{.}
\end{equation}
Note that since $N$ is normal to the level sets of $x^1$ in $\Sigma$, the sets
\begin{equation}
\label{aads_null} \mc{N}_c := \bigcup_{ x^1 (p) = c } ( \operatorname{Im} \Lambda_p )
\end{equation}
form ($\bar{g}$-)null hypersurfaces in an open $\Omega' \subseteq \mc{M}$ that can be obtained by extending $\Sigma$ along the $\Lambda_p$'s, until either the $\Lambda_p$'s form caustics or cut locus points, or the $\Lambda_p$'s terminate (e.g.\ at $\rho = 0$).

We now define our eikonal function $\varphi$ and an adapted set of coordinates $( \sigma, \bar{y} )$ on $\Omega'$ by the property that they are constant along each these geodesics:
\begin{equation}
\label{aads_fcts} \varphi |_{ \Lambda_p } :\equiv x^1 ( \Lambda_p (0) ) \text{,} \qquad \sigma |_{ \Lambda_p } :\equiv \rho ( \Lambda_p (0) ) \text{,} \qquad \bar{y} |_{ \Lambda_p } :\equiv \bar{x} ( \Lambda_p (0) ) \text{,} \qquad p \in \Sigma \text{.}
\end{equation}
In addition, we let $s$ denote (half of) the affine parameter of the $\Lambda_p$'s on $\Omega'$:
\begin{equation}
\label{aads_s} s ( \Lambda_p ( s' ) ) := \tfrac{1}{2} s' \text{,} \qquad p \in \Sigma \text{.}
\end{equation}
Observe that $( \sigma, \bar{y}, s )$ forms a coordinate system on $\Omega'$.

By the first part of \eqref{aads_fcts}, the level sets of $\varphi$ are precisely the null hypersurfaces $\mc{N}_c$'s of \eqref{aads_null}.
Consequently, $\varphi$ satisfies the eikonal equation, and hence the first part of \eqref{conditions_s}:
\[
\bar{g}^{-1} ( d \varphi, d \varphi ) = 0 \text{.}
\]
Furthermore, from \eqref{aads_fcts} and \eqref{aads_s}, one can derive that
\[
\grad_{ \bar{g} } \varphi |_{ \Lambda_p } = \tfrac{1}{2} \Lambda_p' = \tfrac{1}{2} \partial_s \text{,}
\]
which fulfills the second condition in \eqref{conditions_s}. The Figure \ref{fig:const_phi} depicts the above construction. 

\begin{figure}[ht]
\centering
\includegraphics{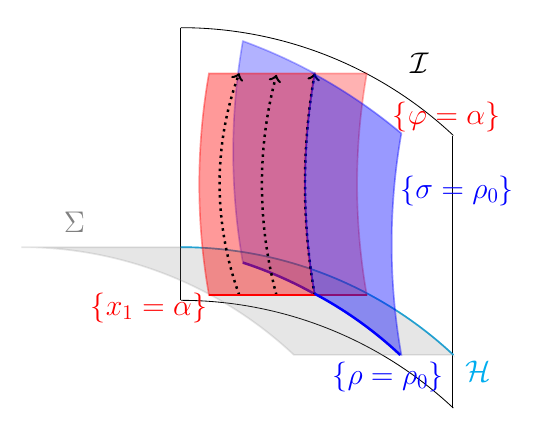}
\caption{Construction of the eikonal function $\textcolor{red}{\varphi}$ and the coordinate \textcolor{blue}{$\sigma$} in the general aAdS case. The choice of a foliation $x_1$ in \textcolor{cyan}{$\mc{H}$}, a spacelike section of $\mc{I}$, allows one to construct null hypersurfaces as level set of a smooth function \textcolor{red}{$\varphi$} by extending along a family of null geodesics (black dotted curves). The level sets of \textcolor{blue}{$\sigma$} are obtained by extending the sets $\Sigma \cap \lbrace \rho = \rho_0 \rbrace$ along the same geodesics.}
\label{fig:const_phi}
\end{figure}

At this point, we make an additional technical assumption on the regularity of our setting:

\begin{assumption} \label{fg_aads_reg}
Let $\Omega \subseteq \Omega'$ such that, identifying $\Omega$ with its image through $( \varphi, \bar{y}, s )$:
\begin{itemize}
\item $\varphi$ lies in $\mc{B}^\infty_2 ( \Omega; \R )$, and both $\sigma^2 \rho^{-2}$, $\sigma^2 V$ lie in $\mc{B}^\infty_0 ( \Omega; \R )$.

\item The components of $\bar{g}$ in the $( \varphi, \bar{y}, s )$-coordinates lie in $\mc{B}^\infty_1 ( \Omega; \R^{ (d+1) \times (d+1) } )$.
\end{itemize}
\end{assumption}

\begin{remark}
Assumption \ref{fg_aads_reg} is stated in manner such that Theorem \ref{Theorem} can be directly applied to our setting.
However, in practice, this $\Omega$ in Assumption \ref{fg_aads_reg} can be obtained by:
\begin{enumerate}
\item Excluding from $\Omega'$ a neighbourhood of the points where the $\Lambda_p$'s degenerate, either from caustics or cut locus points, or by terminating at the conformal boundary $\rho = 0$.

\item Assuming that all the derivatives of $\ms{g}$ in directions along $\mc{I}$ are bounded on $\Omega$.
\end{enumerate}
In particular, if (2) holds, and if $( \mc{M}, g )$ is vacuum, then one can use the EVE and the FG gauge to control $\rho$-derivatives of $\bar{g}$; see \cite{Holzegel22, shao:aads_fg}.
Then, by (1), one can derive that $\bar{g}$ will have the desired $\mc{B}^\infty_1$-regularity in Assumption \ref{fg_aads_reg}.
Furthermore, since $\varphi$ and $V$ are defined from $\bar{g}$, then one can also deduce the desired regularities for $\varphi$, $\sigma^{-2} \rho^2$, and $\sigma^2 V$ in Assumption \ref{fg_aads_reg}.
\end{remark}

\noindent
\textit{Existence of counterexamples.}
We now wish to apply Theorem \ref{Theorem} to $\Omega$, viewed as a subset of $\R^{d+1}$ via $( \sigma, \bar{y}, s )$-coordinates.
As Assumption \ref{fg_aads_reg} implies that $\sigma$ and $\rho$ are comparable, it follows that the level sets of $\sigma$ asymptote to (part of) the timelike conformal boundary $\rho = 0$, which is timelike, hence \eqref{boundedness_metric} holds near $\rho = 0$.
Moreover, Assumption \ref{fg_aads_reg} also implies that
\[
\xi := \sigma^2 \rho^{-2} \big( \mu - \tfrac{ d^2 - 1 }{4} \big) + \sigma^2 V \in \mc{B}^\infty_0 ( \Omega; \R ) \text{.}
\]

As a result, all of the assumptions of Theorem \ref{Theorem} are satisfied, and Theorem \ref{Theorem} now yields a counterexample to unique continuation for $\bar{\mc{P}}_\mu$ from the conformal boundary $\rho = 0$ into $\Omega$.
Finally, reformulating this in terms of $g$ and $\mc{P}_\mu$ yields the following non-uniqueness result:

\begin{corollary} \label{corollary_aads}
Suppose Assumption \ref{fg_aads} holds.
Moreover, let $\varphi$ and $( \sigma, \bar{y}, s )$ be as before, and suppose Assumption \ref{fg_aads_reg} holds.
Then, there exists $\sigma_0 > 0$ and $u, a \in C^\infty ( \Omega \cap \{ \sigma < \sigma_0 \} )$ such that:
\begin{itemize}
\item $u$ is supported on all of $\Omega \cap \{ \sigma < \sigma_0 \}$.

\item $u$ and $a$ both vanish to infinite order as $\rho \searrow 0$.

\item $u$ satisfies the following Klein--Gordon equation on $\Omega \cap \{ \sigma < \sigma_0 \}$:
\begin{equation}
\label{corollary_aads_kg} ( \Box_g + \mu ) u = au \text{.} 
\end{equation}
\end{itemize}
\end{corollary}

In particular, Corollary \ref{corollary_aads} constructs a counterexample to unique continuation for \eqref{kg_aads} on a neighbourhood $\Omega$ of the hypersurface $\Sigma$.
However, we note again that $\Omega$ needs not be local, as it can be enlarged as long as the geodesics $( \Lambda_p )_{ p \in \Sigma }$ do not encounter caustics or terminate.

\begin{remark}
As in the proof of Theorem \ref{Theorem}, both $u$ and $a$ in Corollary \ref{corollary_aads} can be extended to all of $\Omega$ by arbitrarily extending $u$ as a non-vanishing function.
\end{remark}

Observe that, similar to Corollary \ref{corollary_planar_1}, the counterexample $u$ from Corollary \ref{corollary_aads} is only defined on the family $( \Lambda_p )_{ p \in \Sigma }$ of null geodesics emanating from a sector $\Sigma$ of $\mc{I}$.
To have a counterexample that is not limited to a localised sector of $\mc{I}$, one can again apply the trick behind Corollary \ref{corollary_planar_2}, by modifying $\sigma$ as in \eqref{bump_function}, but with the analogue of $B_o$ now lying within the image of $\bar{x}$.

This results in counterexamples $\bar{u}$ that both satisfy an equation of the form \eqref{corollary_aads_kg} and are defined in a slab between two spacelike hypersurfaces in $\mc{M}$.
The details of this, which would depend on the specific spacetime and time foliation being considered, are omitted for brevity.

\medskip
\noindent
\textit{Connections to the GNCC.}
Corollary \ref{corollary_aads} also complements the unique continuation results for \eqref{kg_aads} that were proved in \cite{Shao22, hol_shao:uc_ads, hol_shao:uc_ads_ns, McGill20}.
More specifically, the strongest and most recent result in this direction is \cite[Corollary 5.12]{Shao22}, which roughly states that unique continuation from a portion $\mc{D} \subseteq \mc{I}$ of the conformal boundary (using the natural identification of $\mc{I}$ with $\rho = 0$) holds if $\mc{D}$ satisfies the so-called \emph{generalised null convexity criterion}, or \emph{GNCC}:

\begin{definition} \label{gncc}
$\mc{D}$ satisfies the \emph{GNCC} iff there exists $\eta \in C^4 ( \bar{\mc{D}} )$ such that:
\begin{itemize}
\item $\eta > 0$ on $\mc{D}$.

\item $\eta = 0$ on $\mc{D}$.

\item $( \mf{D}^2 \eta - \eta \mf{g}^{(2)} ) ( \mf{Z}, \mf{Z} ) \gtrsim | \mf{Z} |^2$ for any $\mf{g}^{(0)}$-null vector field $\mf{Z}$ on $\mc{I}$.
\end{itemize}
\end{definition}

For more detailed discussions of the GNCC, see \cite[Definition 3.1]{Shao22}.
(For instance, here $| \mf{Z} |^2$ can be defined using an arbitrary reference Riemannian metric on $\mc{I}$.)

For the present discussion, however, the key point is that the GNCC rules out the existence of the crucial family of geodesics from Corollary \ref{corollary_aads} that is generated by $\varphi$.
This is described more precisely in \cite[Theorem 4.1]{Shao22}, which states that, \emph{assuming $\mc{D}$ satisfies the GNCC, any null geodesic over $\mc{D}$ that is sufficiently near $\mc{D}$ must either start or terminate at the conformal boundary in $\mc{D}$}.
Thus, one can think of $\mc{D}$ as being ``sufficiently large" (with respect to $\mf{g}^{(0)}$) such that the geodesic family $( \Lambda_p )_{ p \in \Sigma }$, and hence the eikonal function $\varphi$, cannot be globally constructed over all of $\mc{D}$.
Conversely, Corollary \ref{corollary_aads} can also be interpreted as saying that counterexamples to unique continuation can be constructed in settings where the GNCC fails to hold.

\begin{remark}
With regards to the above discussion, the GNCC only controls whether null geodesics near the conformal boundary terminate at the boundary.
On the other hand, it does not address whether such null geodesics can form caustics in the interior.
However, it is possible that, similar to the GNCC, the formation of caustics near $\rho = 0$ can also be controlled by the geometry of the conformal boundary, though we do not pursue this question in this paper.
\end{remark}

\begin{remark}
Aside from the GNCC, \cite[Theorem 4.1]{McGill20} also shows that $\mf{g}^{(2)}$ bounds from below the minimal amount of time that null geodesics near the conformal boundary must persist without terminating at the conformal boundary.
Thus, in aAdS settings, one always expects counterexamples to unique continuation from small enough portions of the conformal boundary.
\end{remark}

\raggedright
\raggedbottom

\bibliographystyle{plain}
\bibliography{bibli.bib}

\end{document}